\documentclass[11pt]{amsart}

\usepackage{amsmath,amssymb,amscd,amsthm,amsxtra}
\usepackage[dvips]{graphics,epsfig}

\newtheorem{theorem}{Theorem} [section]
\newtheorem{maintheorem}{Theorem}
\newtheorem{lemma}[theorem]{Lemma}
\newtheorem{proposition}[theorem]{Proposition}
\newtheorem{remark}[theorem]{Remark}

\newtheorem{definition}{Definition}
\newtheorem{corollary}[theorem]{Corollary}

\DeclareMathOperator*{\intt}{\int}
\DeclareMathOperator*{\iintt}{\iint}
\DeclareMathOperator*{\supp}{supp}
\DeclareMathOperator{\med}{med}

\DeclareMathOperator*{\Area}{Area}
\DeclareMathOperator{\MAX}{MAX}

\newcommand{\Z}{\mathbb{Z}}
\newcommand{\R}{\mathbb{R}}

\newcommand{\T}{\mathbb{T}}

\newcommand{\al}{\alpha}
\newcommand{\dl}{\delta}

\newcommand{\eps}{\varepsilon}

\newcommand{\g}{\gamma}
\newcommand{\G}{\Gamma}
\newcommand{\ld}{\lambda}

\newcommand{\s}{\sigma}
\newcommand{\ft}{\widehat}
\newcommand{\wt}{\widetilde}

\newcommand{\dx}{\partial_x}
\newcommand{\dt}{\partial_t}
\newcommand{\dd}{\partial_\dl}
\newcommand{\invft}[1]{\overset{\vee}{#1}}

\newcommand{\col} [2] {
\left( \begin{smallmatrix} #1 \\ #2 \end{smallmatrix} \right)}

\newcommand{\jb}[1]
{\langle #1 \rangle}

\newcommand{\mtx}[4]
{\left( \begin{smallmatrix} #1 &  #2 \\ #3 &  #4 \end{smallmatrix} \right)}

\allowdisplaybreaks[2]

\begin{document}

\title
[Diophantine Condition in LWP Theory of KdV Systems]
{\bf Diophantine Conditions in Well-Posedness Theory of Coupled KdV-Type Systems:  Local Theory}

\author{Tadahiro Oh}

\address{Tadahiro Oh\\
Department of Mathematics\\
University of Toronto\\
40 St. George St, Rm 6290,
Toronto, ON M5S 2E4, Canada}

\email{oh@math.toronto.edu}



\begin{abstract}
We consider the local well-posedness problem of a one-parameter family of  coupled KdV-type systems
both in the periodic and non-periodic setting. 
In particular, we show that certain resonances occur, closely depending on the value of a coupling parameter $\al$ when $\al \ne 1$. 
In the periodic setting, we use the Diophantine conditions to characterize the resonances, 
and establish sharp local well-posedness of the system in $H^s(\T_\ld), s \geq s^\ast$,
where $s^\ast = s^\ast(\al) \in (\frac{1}{2}, 1]$ is determined by the Diophantine characterization of certain constants derived from the coupling parameter $\al$.
We also present a sharp local (and global) result in $L^2(\R)$.
In the appendix, we briefly discuss the local well-posedness result in $H^{-\frac{1}{2}}(\T_\ld) $ for $\al= 1$ without the mean 0 assumption,
by introducing the vector-valued $X^{s, b}$ spaces.
\end{abstract}

\maketitle

\section{Introduction}

In this paper, we consider the local well-posedness (LWP) in both periodic and non-periodic settings
of coupled KdV systems of the form:
\begin{equation} \label{KDVsystem1}
\begin{cases}
u_t + a_{11} u_{xxx} + a_{12} v_{xxx} +  b_1 u u_x + b_2 u v_x + b_3 u_x v + b_4 v v_x = 0 \\ 
v_t + a_{21} u_{xxx} + a_{22} v_{xxx} + b_5 u u_x + b_6 u v_x + b_7 u_x v + b_8 v v_x  = 0 \\
(u, v) \big|_{t = 0} = (u_0, v_0) , 
\end{cases}
\end{equation}

\noindent
where $A = \bigl(\begin{smallmatrix}
a_{11} & a_{12} \\ 
a_{21} & a_{22}
\end{smallmatrix} \bigr)$
is self-adjoint and non-singular, and $u$ and $v$ are real-valued functions.
There are several systems of this type: the Gear-Grimshaw system \cite{GG}, the Hirota-Satsuma system \cite{HS}, 
the Majda-Biello system \cite{MB}, etc.
Now, write $A$ as  $A = M^{-1} D M$ where $D = \big(\begin{smallmatrix} d_1 & 0 \\ 0 & d_2 \end{smallmatrix}\big)$
with $d_j \in \R\setminus \{0\}$ and $M$ is orthogonal.  
Then, by letting
$ M  \big(\begin{smallmatrix} u\\v \end{smallmatrix}\big) (x, d_1^{-1} t) 
\mapsto \big(\begin{smallmatrix} u \\v \end{smallmatrix}\big) (x, t)$,
one can reduce \eqref{KDVsystem1} to 
\begin{equation} \label{KDVsystem2}
\begin{cases}
u_t +  u_{xxx} + \wt{b_1} u u_x + \wt{b_2} u v_x + \wt{b_3} u_x v + \wt{b_4} v v_x  = 0 \\ 
v_t + \al v_{xxx} + \wt{b_5} u u_x + \wt{b_6} u v_x + \wt{b_7} u_x v + \wt{b_8} v v_x  = 0 \\
(u, v) \big|_{t = 0} = (u_0, v_0),  
\end{cases}
\end{equation}

\noindent
where $\al  = \frac{d_2}{d_1} \in \R\setminus \{0\}$, 
$(x, t) \in \T_\ld\times\R$ or $\R \times\R$,
with $\T_\ld = [0, 2\pi\ld)$ for some $\ld > 0$.
Note that we do not consider the case $\al = 0$ (i.e. $d_2 = 0$) or  $\al =\pm\infty$ (i.e. $d_1 = 0$), 
since \eqref{KDVsystem2} is not dispersive in those cases.

When $\al = 1$, the basic techniques developed for the LWP of KdV by Kenig-Ponce-Vega \cite{KPV4} can be applied to \eqref{KDVsystem2}
as discussed below, thus yielding the LWP of  \eqref{KDVsystem2}  in $H^{-\frac{1}{2}}(\T_\ld) \times H^{-\frac{1}{2}}(\T_\ld)$
and $H^{-\frac{3}{4}+}(\R) \times H^{-\frac{3}{4}+}(\R)$.
However, when $\al \ne 1$, we show that there is an interval $I_0$ 
such that particular resonances occur for $\al \in I_0 \setminus \{1\}$
and that the regularities for the LWP in both periodic and non-periodic settings
 need to be much higher than those for $\al = 1$.
In particular, 
we use the Diophantine conditions
 to quantify this regularity in the periodic setting.

As a model example, we consider the local well-posedness (LWP) problem of the following system:
\begin{equation} \label{MB}
\begin{cases}
u_t + u_{xxx} + vv_x = 0\\
v_t + \al v_{xxx} + (uv)_x = 0 \\
(u, v) \big|_{t = 0} = (u_0, v_0),  
\end{cases}
 \ (x, t) \in \T_\ld\times\R \text{ or } \R\times\R,
\end{equation}

\noindent
where $ \T_\ld = [0, 2\pi\ld)$ with $\ld \geq 1$, $0< \alpha \leq1$, and $u$ and $v$ are real-valued functions.
We consider the Cauchy problem \eqref{MB} with $(u_0, v_0)  \in H^s(\T_\ld) \times H^s(\T_\ld)$ or $ H^s(\mathbb{R}) \times H^s(\mathbb{R})$. 
While we state and prove our results only for $0 < \al \leq 1$ for the physical reason mentioned below, 
we discuss the corresponding results for other values of $\al \ne 0$.
In particular, we have $I_0$ = $(0, 4]$ for \eqref{MB}. 
See  Remark \ref{JJREMARK6}.

The system \eqref{MB} has been proposed by Majda and Biello \cite{MB} 
as a reduced asymptotic model to study the nonlinear resonant interactions of 
long wavelength equatorial Rossby waves and barotropic Rossby waves with a significant mid-latitude projection,
in the presence of suitable horizontally and vertically sheared zonal mean flows.
We henceforth  refer to \eqref{MB} as  {\it the Majda-Biello System}.
In \cite{MB},  the values of $\al $ are numerically determined and they are
$0.899$, $0.960$, and $0.980$ for different equatorial Rossby waves.
Of particular interest to us is the periodic case because of more challenging mathematical nature of the periodic problem 
as well as the physical relevance of the proposed model (the spatial period for the system before scaling is set as $40, 000$ km in \cite{MB}.)

First, we review the recent well-posedness results for the Korteweg-de Vries (KdV) equation:
\begin{equation} \label{KDV}
\left\{
\begin{array}{l}
u_t + u_{xxx} + u u_x = 0\\
u(x, 0) = u_0(x) \in H^s{(\T)} \text{ or } H^s{(\mathbb{R})}.
\end{array}
\right.
\end{equation}

\noindent
Bourgain \cite{BO1} proved the LWP in $L^2(\T)$ and $L^2(\R)$
via the contraction mapping principle.
Kenig-Ponce-Vega \cite{KPV4} improved Bourgain's result and established the LWP in $H^{-\frac{1}{2}}(\T)$ and $H^{-\frac{3}{4}+}(\R)$.
Colliander-Keel-Staffilani-Takaoka-Tao \cite{CKSTT4} proved the corresponding global well-posedness results via the $I$-method. 
More recently, Christ-Colliander-Tao \cite{CCT} proved the LWP in $H^{-\frac{3}{4}}(\R)$ via the modified Miura transform and the LWP of the modified KdV,
and
Kappeler-Topalov \cite{KT} proved the global well-posedness of the KdV in $H^{-1}(\T)$, 
using the complete integrability of the equation. (See \cite{CKSTT4} and the references therein.)

In this paper,  we use the contraction mapping principle to show the LWP of \eqref{MB}, following the basic arguments in 
\cite{BO1}, \cite{KPV4}, \cite{CKSTT4}.
Before discussing how the KdV theory immediately yields the LWP of \eqref{MB} (and \eqref{KDVsystem2}) when $\al = 1$, 
we list some basic properties of the Majda-Biello system \eqref{MB}.
Several conservation laws are known for the system:
\[ E_1 = \int u \, dx, \quad E_2 = \int v \, dx, \quad E_3 = \int u^2 + v^2 dx, \quad E_4 =\frac{1}{2} \int u_x^2 + \al v_x^2 - u v^2 dx, \] 

\noindent
where $E_4$ is the Hamiltonian of the system. No other conservation laws 
seem apparent; whether the Majda-Biello system  \eqref{MB} is completely integrable is unknown.
Also, the system has scaling which is similar to that of KdV and the critical Sobolev index $s_c$ is $-\frac{3}{2}$ just like KdV.

In \cite{BO1}, Bourgain  introduced a new weighted space-time Sobolev space $X^{s, b}$
whose norm is given by
\begin{equation} \label{Xsb}
\| u \|_{X^{s, b}(Z \times \mathbb{R})} = \big\| \jb{\xi}^s \jb{\tau - \xi^3}^b \ft{u}(\xi, \tau) \big\|_{L^2_{\xi, \tau}(Z^\ast \times \R)},
\end{equation}

\noindent
where $\jb{ \: \cdot \:} = 1 + |  \cdot  | $, and the spatial Fourier domain $Z^\ast = \R$ if the spatial domain $Z = \R$ and $Z^\ast = \Z/\ld$ if $Z = \T_\ld$.
Kenig-Ponce-Vega \cite{KPV4} proved that the sharp bilinear estimate
\begin{equation} \label{KPVbilinear1}
\| \dx (u v)  \|_{X^{s, b -1}(\mathbb{R} \times \mathbb{R})} \lesssim \|u\|_{X^{s, b}}\|v\|_{X^{s, b}} 
\end{equation}
holds for $s > -\frac{3}{4}$ and $b = \frac{1}{2}+$. 
Let $S(t) = e^{-t \dx^3}$ and $\eta(t)$ be a smooth cutoff function supported on $[-2, 2]$ whose value is 1 on $[-1, 1]$.
Then, along with the linear estimates \cite{KPV6}:
\begin{equation*} 
\begin{cases}
\vphantom{\bigg|}\| \eta(t)S(t) u_0 \|_{X^{s, b}(\R\times \R)} \lesssim \| u_0 \|_{H^s_x(\R)} \\
\Big\| \eta (t) \int_0^t S(t-t') N\big( u ( t') \big) dt' \Big\|_{X^{s, b} (\R \times \R)}
\lesssim \left\| N ( u ) \right\|_{X^{s, b-1},}
\end{cases}
\end{equation*}

\noindent
the bilinear estimate \eqref{KPVbilinear1} and scaling establish the LWP of KdV \eqref{KDV} in $H^s(\mathbb{R})$ for $s > -\frac{3}{4}$.
Clearly, the same argument applies to the non-periodic Majda-Biello system \eqref{MB} when $\al = 1$.  
Hence, when $\al = 1$, \eqref{MB} is locally well-posed in $H^s(\mathbb{R}) \times H^s(\R)$ for $s > -\frac{3}{4}$.

In the periodic case, we have a weaker linear estimate on the Duhamel term \cite{BO1}, \cite{CKSTT4}:
\begin{equation*} 
\left\| \eta (t) \int_0^t S(t-t') N\big( u ( t') \big) dt' \right\|_{X^{s, \frac{1}{2}}(\T_\ld \times \mathbb{R})}
\lesssim \left\| N ( u ) \right\|_{Z^s (\T_\ld \times \mathbb{R}),}
\end{equation*}

\noindent
where 
$ \| u \|_{Z^s(\T_\ld \times \mathbb{R})} = \| u \|_{X^{s, -\frac{1}{2}} (\T_\ld \times \mathbb{R})
}+ \|  \jb{\xi}^s \jb{\tau - \xi^3}^{-1}\ft{u}(\xi, \tau) \|_{L^2_{\xi}L^1_\tau (\Z/\ld \times \mathbb{R})}. $
This is due to the fact that $b = \frac{1}{2}$.  
Strictly speaking,  we need to use 
$\| u \|_{Y^s} = \| u \|_{X^{s, \frac{1}{2} }} + \| \jb{\xi}^s \ft{u}(\xi, \tau) \|_{L^2_{\xi}L^1_\tau}$,
instead of $X^{s, \frac{1}{2}}$ norm, 
in order to control the $C( [-T, T]; H^s (\T_\ld) )$ norm of solutions.
However, this modification does not cause a serious problem, 
and 
we  use $X^{s, \frac{1}{2} }$ norm for simplicity in the current discussion. 
The crucial bilinear estimate for the LWP of KdV  is then: 
\begin{equation} \label{KPVZbilinear}
\| \dx (u v)  \|_{Z^s(\T_\ld \times \mathbb{R} )} \lesssim \ld^{0+}\|u\|_{X^{s, \frac{1}{2}}}\|v\|_{X^{s, \frac{1}{2}}.} 
\end{equation}
Assuming the mean 0 condition on $u$ and $v$, it is shown in \cite{KPV4} and  \cite{CKSTT4} that the bilinear estimate
\eqref{KPVZbilinear}
 holds for $ s \geq -\frac{1}{2}$ with $b = \frac{1}{2}$
and fails for any $b \in \mathbb{R} $ if  $s < -\frac{1}{2}$.
A key ingredient is the algebraic identity
\begin{equation} \label{P1algebra}
\xi^3 - \xi_1^3 - \xi_2^3 = 3 \xi \xi_1\xi_2,    \qquad \rm{for} \qquad  \xi = \xi_1 + \xi_2.
\end{equation}

When $\al = 1$, the same estimates can be applied to the periodic Majda-Biello system \eqref{MB} with the mean 0 assumption.
Hence, along with the conservation of the means $E_1$ and $E_2$, it follows that, when $\al = 1$, \eqref{MB} is locally well-posed in 
$H^{-\frac{1}{2}} (\T_\ld)\times H^{-\frac{1}{2}} (\T_\ld)$ for the mean 0 initial conditions.
We address the  LWP theory {\it without} the mean 0 assumption in the appendix.

Now, let's turn to the case $0<\al<1$.
In this case, we have two linear semigroups $S(t)= e^{ -t \dx^3}$ and $S_\al(t)= e^{ -\al t \dx^3}$ corresponding to the linear equations for $u$ and $v$,
respectively.
Thus, the Fourier transform of the solutions to the linear  equation for $u$ is supported on $\{ \tau = \xi^3\}$ and 
that to the linear equation for $v$ is supported on $\{ \tau = \al \xi^3\}$. 
Since $\al \ne 1$, they are supported on {\it distinct} curves, and this causes nontrivial resonance interactions
which are not present when $\al = 1$.
It turns out that the Cauchy problem \eqref{MB} is locally well-posed in $H^s(\T_\ld) \times  H^s(\T_\ld)$
for $s\geq s^\ast = s^\ast(\al)$, where $s^\ast(\al) $ ranges over $ (1/2, 1]$ depending on the value of the coupling parameter $\al$.
In order to precisely quantify $s^\ast$ in terms of $\al$, we need to go through some preliminaries.

For $\al \ne 1$,  we need to define two distinct $X^{s, b}$ spaces. 
For $s, b \in \mathbb{R}$, let $X^{s, b}(\mathbb{T}_\ld\times\mathbb{R})$ and $X_\al^{s, b}(\mathbb{T}_\ld\times\mathbb{R})$ be the completion of 
the Schwartz class $\mathcal{S} (\mathbb{T} \times \mathbb{R})$ with respect to the norms
\begin{align} 
\|u\|_{X^{s, b}(\mathbb{T_\ld} \times \mathbb{R})} 
& = \big\| \jb{\xi}^s \jb{\tau - \xi^3}^b \ft{u}(\xi, \tau) \big\|_{L^2_{\xi, \tau}(\mathbb{Z/\ld} \times \mathbb{R})} 
\\ \|v\|_{X_{\al}^{s, b}(\mathbb{T_\ld} \times \mathbb{R})} 
& = \big\| \jb{\xi}^s \jb{\tau - \al \xi^3}^b \ft{v}(\xi, \tau) \big\|_{L^2_{\xi, \tau}(\mathbb{Z/\ld} \times \mathbb{R})} . 
\end{align}

\noindent
The study of the periodic Majda-Biello system \eqref{MB} will be based  on the iteration in the spaces $X^{s, \frac{1}{2}}\times X_{\al}^{s, \frac{1}{2}}$.

For the following argument, assume $ \ld = 1$ for simplicity.
(For general $\ld \geq 1$, the implicit constants in \eqref{bilinear1} and \eqref{bilinear2} below depend on $\ld$.
See Propositions \ref{PROP:bilinear1} and \ref{PROP:bilinear2}.)
From the standard estimates on the Duhamel terms, the LWP of \eqref{MB} follows once we establish the following bilinear estimates:
\begin{align} \label{bilinear1}
\| \dx (v_1 v_2)  \|_{X^{s, b-1} (\mathbb{T} \times \mathbb{R} )} & \lesssim \|v_1\|_{X_\al^{s, b}(\mathbb{T} \times \mathbb{R} )}
\|v_2\|_{X_\al^{s, b}(\mathbb{T} \times \mathbb{R} )} 
\\ \label{bilinear2}
\| \dx (u v)  \|_{X_\al^{s, b-1} (\mathbb{T} \times \mathbb{R} )} & \lesssim \|u\|_{X^{s, b}(\mathbb{T} \times \mathbb{R} )}
\|v\|_{X_\al^{s, b}(\mathbb{T} \times \mathbb{R} ).} 
\end{align}

\noindent
First, consider the first bilinear estimate \eqref{bilinear1}.  
As in the KdV case \cite{BO1}, \cite{KPV4}, \cite{CKSTT4}, we define the bilinear operator $\mathcal{B}_{s, b} (\cdot, \cdot)$ by
\[ \mathcal{B}_{s, b} (f, g) (\xi, \tau) =  \frac{\xi \jb{\xi}^s }{\jb{\tau - \xi^3}^{1-b}}\frac{1}{2\pi}
\sum_{\xi_1+ \xi_2 = \xi} \intt_{\tau_1 + \tau_2 = \tau} \frac{f(\xi_1, \tau_1) g(\xi_2, \tau_2)}{\jb{\xi_1}^s 
\jb{\xi_2}^s \jb{\tau_1 - \al \xi_1^3}^b  \jb{\tau_2 - \al \xi_2^3}^b} d\tau_1.
\]
Then, \eqref{bilinear1} holds if and only if
\begin{equation} \label{c_1dual} 
\left\| \mathcal{B}_{s,b}(f, g) \right\|_{L^2_{\xi, \tau}} \lesssim \| f \|_{L^2_{\xi, \tau}} \| g \|_{L^2_{\xi, \tau}} . 
\end{equation}

\noindent
As in the KdV case,  $\dx$ appears on the left hand side of \eqref{bilinear1} and thus 
we need to make up for this loss of derivative from 
$\jb{\tau - \xi^3}^{1-b} \jb{\tau_1 - \al \xi_1^3}^b  \jb{\tau_2 - \al \xi_2^3}^b$ in the denominator.
Recall that we basically gained $\frac{3}{2}$ derivatives in the KdV case (at least for $b = \frac{1}{2}$ with $\xi, \xi_1, \xi_2 \ne 0$) 
thanks to the algebraic identity \eqref{P1algebra}.
However, when $\al \ne 1$, we no longer have such an identity and 
we have
\begin{align} \label{resonance1}
\max \big(  \jb{\tau & - \xi^3},    \jb{\tau_1 - \al \xi_1^3},   \jb{\tau_2 - \al \xi_2^3} \big) 
\sim \jb{\tau - \xi^3}+  \jb{\tau_1 - \al \xi_1^3}+   \jb{\tau_2 - \al \xi_2^3} 
\notag \\
& \gtrsim \big| (\tau - \xi^3) - (\tau_1 - \al \xi_1^3) - (\tau_2 - \al \xi_2^3) \big|  
 = | \xi^3 -  \al \xi_1^3 - \al \xi_2^3 | ,  
\end{align}
where $\xi = \xi_1 + \xi_2$ and  $\tau = \tau_1 + \tau_2$.
Note that the last expression in \eqref{resonance1} can be 0 for infinitely many (nonzero) values of $\xi, \: \xi_1,$ and $ \xi_2$,
causing {\it resonances}. 
By solving {\it the resonance equation}:
\begin{equation} \label{JJreseq1}
\xi^3 - \al \xi_1^3 - \al \xi_2^3 = 0 \text{ with }  \xi = \xi_1 + \xi_2, 
\end{equation} 
we have $( \xi_1, \xi_2) = (c_1 \xi, c_2 \xi)$ or $(c_2 \xi, c_1 \xi)$, where 
\begin{equation} \label{c_1}  
c_1 = \tfrac{1}{2} + \tfrac{\sqrt{-3 + 12 \al^{-1}}}{6} \ \text{ and } \
c_2 = \tfrac{1}{2} - \tfrac{\sqrt{-3 + 12 \al^{-1}}}{6}.
\end{equation}
Note that $c_1, c_2 \in \mathbb{R}$ (with $ \ c_1 + c_2 = 1$)  
 if and only if $0 < \al \leq 4.$

If $c_1 \in \mathbb{Q}$ (and thus $c_2 \in \mathbb{Q}$), then there are infinitely many values of $\xi \in \mathbb{Z}$ 
such that  $c_1\xi, \: c_2 \xi \in \mathbb{Z}$.  This causes resonances for infinitely many values of $\xi$, and 
thus we do {\it not} expect any gain of derivative from $\jb{\tau - \xi^3} \jb{\tau_1 - \al \xi_1^3}  \jb{\tau_2 - \al \xi_2^3}  $ in this case.

If $c_1 \in \mathbb{R} \setminus \mathbb{Q}$,  then $c_1 \xi$ is not an integer for any $\xi \in \mathbb{Z}$.
Thus,  $\xi - \al \xi_1^3 - \al \xi_2^3 \ne 0 $ for any $\xi, \: \xi_1, \xi_2 \in \mathbb{Z}$. 
However,  generally speaking, $\xi - \al \xi_1^3 - \al \xi_2^3$ can be  arbitrarily close to 0,
since $c_1 \xi$ can be arbitrarily close to an integer.
Therefore, we need  to measure {\it how ``close" $c_1$ is to rational numbers}.
We'll use the following definition regarding the Diophantine conditions commonly used in dynamical systems.

\begin{definition} [{Arnold \cite{AR}}]
A real number $\rho$ is called of type $(K, \nu)$ \textup{(}or simply of type ${\nu}$\textup{)} if there exist positive $K$ and $\nu$ such that 
for all pairs of integers $(m, n)$, we have
\begin{equation} \label{lowerbd}
\left| \rho - \frac{m}{n} \right| \geq \frac{K}{ |n|^{2+\nu}} .
\end{equation}
\end{definition}

\noindent
Also, for our purpose, we define {\it the minimal type index} of a given real number $\rho$.
\begin{definition}
Given a real number $\rho$, define the minimal type index ${\nu_{\rho}}$ of ${\rho}$ by
\[ \nu_{\rho} = \begin{cases}
\infty \text{, if } \rho \in \mathbb{Q} \\
\inf \{ \nu > 0 : \rho \text{ is of type } \nu \} \text{, if } \rho \notin \mathbb{Q} 
\end{cases} \]
\end{definition}

\begin{remark} \label{JJDIO} \rm
Then, by Dirichlet Theorem {\cite[p.112]{AR}} and {\cite[p.116, lemma 3]{AR}},
it follows that $\nu_\rho \geq 0$ for {\it  any} $\rho \in \mathbb{R}$ and  $\nu_\rho = 0$ for {\it almost every } $\rho \in \mathbb{R}$. 
\end{remark}

Using the minimal type index of $c_1$, we  prove, in Propositions \ref{PROP:bilinear1} and \ref{PROP:counterexample1},  
the bilinear estimate \eqref{bilinear1} 
holds for $s \geq \min ( 1, \frac{1}{2} + \frac{1}{2} \nu_{c_1} + ) $ with $ b = \frac{1}{2}$, 
and fails for  $b \in \R$ if $s < \min ( 1, \frac{1}{2} + \frac{1}{2} \nu_{c_1}  )$.

Let us now turn to the second bilinear estimate \eqref{bilinear2}.
 In this case, the resonance equation is given by
\begin{equation} \label{resonance2}
 \al \xi^3 - \xi_1^3 - \al \xi_2^3 = 0 \text{ with } \xi = \xi_1 + \xi_2.
\end{equation}

\noindent By solving \eqref{resonance2}, we obtain
$(\xi_1, \xi_2) = \big(d_1 \xi, (1-d_1) \xi\big),  \big(d_2 \xi, (1-d_2) \xi\big), (0, \xi)$, where 
\begin{equation} \label{d_1 and d_2}
d_1 = \tfrac{- 3 \al + \sqrt{3 \al(4 - \al)}}{2 (1 - \al)} \ \text{ and } 
\ d_2 = \tfrac{- 3 \al - \sqrt{3 \al(4 - \al)}}{2 (1 - \al)}.
\end{equation}
Note that $d_1, d_2 \in \mathbb{R}$ if and only if $\al \in [ 0, 1) \cup (1, 4]$. 

Using the minimal type indices $\nu_{d_1}$ and $\nu_{d_2}$, we  prove, in Propositions \ref{PROP:bilinear2} and \ref{PROP:counterexample2},  
the bilinear estimate \eqref{bilinear2} 
holds for $s \geq \min ( 1, \frac{1}{2} + \frac{1}{2} \max( \nu_{d_1} ,  \nu_{d_2})+ ) $ with $ b = \frac{1}{2}$
under the mean 0 assumption on $u$,
and fails for any $b \in \R$ if $s < \min ( 1, \frac{1}{2} + \frac{1}{2} \max( \nu_{d_1} ,  \nu_{d_2}) )$.
Note that we need to take the maximum of $\nu_{d_1}$ and $\nu_{d_2}$ 
since we have $d_1 + d_2 \notin \mathbb{Q}$ in general, unlike $c_1 + c_2 = 1$.

Finally, we  state the main local well-posedness result.
Let $s_0 = \frac{1}{2} + \frac{1}{2} \max( \nu_{c_1}, \nu_{d_1},  \nu_{d_2})$.
Note that $s_0 = \frac{1}{2}$ for almost every $\alpha \in (0, 1)$
in view of Remarks \ref{JJDIO} and \ref{JJREMARK2}.

\begin{maintheorem} \label{THM:YLWP} 
Let $0 < \al < 1$ and $\ld \geq 1$. Assume the mean 0 condition on $u_0$.
Then, the Majda-Biello system \eqref{MB}  is locally well-posed in $H^s(\mathbb{T}_\ld) \times H^s(\mathbb{T}_\ld)$
for 
$s \geq s^\ast(\al) := \min ( 1,s_0 + ). $ 
In particular, for almost every $\al \in (0, 1)$, the Majda-Biello system is locally well-posed in 
$H^s(\mathbb{T}_\ld) \times H^s(\mathbb{T}_\ld)$ for $s > \frac{1}{2} $.
Moreover, the (real-valued) Majda-Biello system is globally well-posed
in $H^1(\mathbb{T}_\ld) \times H^1(\mathbb{T}_\ld) $
due to the conservation of the Hamiltonian.
\end{maintheorem}

This result is sharp in the following sense.
Consider the solution map $\Phi_t: (u_0, v_0)  \in H^s \times H^s 
\longmapsto \big( u(t), v(t) \big) \in H^s \times H^s$ for $|t| \ll 1$.
Then, we have the following necessary conditions on the Sobolev index $s$ 
(This kind of result is often referred to as {\it ill-posedness} results. c.f. \cite{KPV5}, \cite{TZ}.)

\begin{maintheorem} \label{THM:illposedonT}
\textup{(a)} If $\max ( \nu_{c_1}, \nu_{d_1}, \nu_{d_2}) > 1$ and the solution map is $C^2$, then $s \geq 1.$
In particular, this result applies to the case when any of $c_1, d_1,$ or $ d_2$ is a rational number.

\noindent
\textup{(b)} If $\max ( \nu_{c_1}, \nu_{d_1}, \nu_{d_2}) \leq 1$ and the solution map is $C^3$, then 
$s \geq s_0$.
\end{maintheorem}

Therefore, if we require that the solution map is smooth (at least $C^3$), then the LWP theory on $\T_\ld$ is complete
(except for the endpoint case $s = \frac{1}{2} + \frac{1}{2} \max ( \nu_{c_1}, \nu_{d_1}, \nu_{d_2})$ 
when $\max ( \nu_{c_1}, \nu_{d_1}, \nu_{d_2}) < 1$ which leads to a number-theoretic question. See Remark \ref{JJREMARK3}.)

We'd like to point out the following. 
On the one hand the regularity index $s^\ast(\al)$ for the local well-posedness is $\frac{1}{2}+$ for almost every $\al \in (0, 1)$. 
On the other hand,  for any interval $I \subset (0, 1)$, there exists $\al \in I$ such that $s^\ast(\al) = 1$.
This shows that the well-posedness (below $H^1$) of the periodic Majda-Biello system  \eqref{MB}
is very unstable under a slightest perturbation of the parameter $\al$.
Also, note that $\al = 1$ is a special value since \eqref{MB} is LWP for $s \geq -\frac{1}{2}$ when $\al = 1$.

Lastly, note that we have stated Theorem \ref{THM:YLWP} on $\mathbb{T}_\ld$ for $\ld \geq 1$
since in proving the global well-posedness on $\mathbb{T}$ (i.e. $\ld = 1$) via the I-method, 
we need to have LWP for $\ld \geq 1$ due to the scaling argument.  See \cite{OH2}. 
However, Theorem \ref{THM:YLWP} itself holds on $\mathbb{T}_\ld$ for any $\ld > 0$
as long as $\ld$ is fixed.

\medskip

Now, let's discuss the LWP in the non-periodic setting for $0 < \al < 1$.
In this case, the LWP of  \eqref{MB} follows once we prove the bilinear estimates:
\begin{align} \label{Rbilinear1}
\| \dx (v_1 v_2)  \|_{X^{s, b-1} ( \mathbb{R}^2 )} & \lesssim \|v_1\|_{X_\al^{s, b}( \mathbb{R}^2 )}\|v_2\|_{X_\al^{s, b}( \mathbb{R}^2 )} 
\\ \label{Rbilinear2}
\| \dx (u v)  \|_{X_\al^{s, b-1} ( \mathbb{R}^2 )} & \lesssim \|u\|_{X^{s, b}( \mathbb{R}^2 )}\|v\|_{X_\al^{s, b}( \mathbb{R}^2 )} .
\end{align}

\noindent
As in the periodic case, we obtain two resonance equations
$  \xi^3 -  \al \xi_1^3 - \al \xi_2^3 = 0 $
and 
$ \al \xi^3 - \xi_1^3 - \al \xi_2^3 = 0$
with $\xi = \xi_1 + \xi_2$,
giving rise to $c_1$, $d_1$, and $d_2$.
Since the spatial Fourier variable $\xi$ is not discrete in this case,
the rational/irrational character of $c_1, d_1,$ and  $d_2$ is irrelevant. 
In Propositions \ref{PROP:Rbilinear} and \ref{PROP:Rcounterexample}, 
we  prove the sharp bilinear estimates \eqref{Rbilinear1} and  \eqref{Rbilinear2} for $s \geq 0$ with some $b = \frac{1}{2}+$. 
Thus, we obtain

\begin{maintheorem} \label{THM:LWPonR}
Let $0<\al <1$.
The Majda-Biello system \eqref{MB}   is locally well-posed in $H^s(\mathbb{R}) \times H^s(\mathbb{R})$ for $s \geq 0$. 
Moreover, the (real-valued) Majda-Biello system is globally well-posed
in $L^2(\mathbb{R}) \times L^2(\mathbb{R}) $
due to the $L^2$-conservation law.
\end{maintheorem}

\noindent
Moreover, this result is sharp in the following sense:
\begin{maintheorem} \label{THM:illposedonR}
 If the solution map is $C^2$, then 
$s \geq 0$.
\end{maintheorem}

\noindent 
Therefore, if we require that the solution map is smooth (at least $C^2$), 
then the well-posedness theory on $\mathbb{R}$ is complete.
Note that although the rational/irrational character of $c_1, d_1,$ and  $d_2$ is irrelevant in this case, 
the result for $ \al  \in (0, 1)$ is much worse than that for $\al = 1$,
where the threshold is $s = -\frac{3}{4}$.

\begin{remark} \rm
There are several known results for \eqref{KDVsystem1}.
In \cite{BPST}, Bona-Ponce-Saut-Tom proved the LWP (for the Gear-Grimshaw system)  in $H^s(\mathbb{R}) \times H^s(\mathbb{R})$ for $ s > \frac{3}{4}$.  
Then, via the $X^{s, b}$ space and the bilinear estimate \eqref{KPVbilinear1} 
in the non-periodic setting, 
further LWP and partial GWP results were proven in Ash-Cohen-Wang \cite{ACW}, Feng \cite{F}, 
 Linares-Panthee \cite{LP}, Saut-Tzvetkov \cite{ST} (also for KP systems on $\mathbb{R}^2$),
where the best LWP result is proven in $H^s(\mathbb{R}) \times H^s(\mathbb{R})$ for $ s > -\frac{3}{4}$.
We'd like to point out that in \cite{ACW}, \cite{LP}, and \cite{ST},  it was  assumed that 
$A = \bigl(\begin{smallmatrix}
a_{11} & a_{12} \\ 
a_{21} & a_{22}
\end{smallmatrix} \bigr)$
can be reduced to the identity matrix
$\bigl(\begin{smallmatrix}
1 & 0\\0& 1
\end{smallmatrix} \bigr)$
by similarity transformations and scaling changes.
However, in general,  $A$ can only be reduced to $\bigl(\begin{smallmatrix}
1 & 0\\0& \al
\end{smallmatrix} \bigr)$ even if we assume that $A$ is self-adjoint and non-singular.
Thus their results do {\it not} apply to the general KdV system \eqref{KDVsystem1}, contrary to their claim.
This is due to the fact that separate space-time scalings are applied to $u$ and $v$ in their work, 
which causes rescaled functions $\wt{u}$ and $\wt{v}$ to be evaluated at different space-time points $(x, t)$.
In this case, a simple application of the bilinear estimate \eqref{KPVbilinear1} is clearly prohibited.
As one can see from Theorems \ref{THM:LWPonR} and \ref{THM:illposedonR}, 
one can in general expect LWP on $\mathbb{R}$ only for $ s\geq 0$ if $\al \ne 1$.
Also, see Alvarez-Samaniego and Carvajal \cite{AC}.
\end{remark}

In the appendix, by introducing the vector-valued $X^{s, b}$ spaces, we  sketch the proof of the following theorem:
\begin{maintheorem} \label{MTHM:mean0LWP} 
Let $\al =1$ and $\ld \geq 1$.
The Majda-Biello system \eqref{MB}   is locally well-posed in $H^{s}(\mathbb{T}_\ld) 
\times H^{s}(\mathbb{T}_\ld)$ for $s \geq -\frac{1}{2}$
\textup{(}without the mean 0 assumption on $u$ and $v$.\textup{)}
\end{maintheorem}

This work is a part of the author's Ph.D. thesis \cite{OHTHESIS} at the University of Massachusetts, Amherst.
In two forthcoming papers on this subject,
we address the global well-posedness via the $I$-method \cite{OH2}, and 
the invariant measures (i.e. the Gibbs measure as a weighted Wiener measure
for a.e. $\al \in (0, 4)$ in \cite{OH3} and the white noise for $\al = 1$ in \cite{OH4}) 
and  
the almost surely global well-posedness on the statistical ensembles in the periodic setting.

This paper is organized as follows:
In Section 2, we introduce some standard notations.
In Section 3, we make the precise statements of the bilinear estimates \eqref{bilinear1} and \eqref{bilinear2} on $\T_\ld$ for $0< \al < 1$
along with the basic linear estimates.
Then, we prove Theorem \ref{THM:YLWP}.
In Section 4, we prove the sharp bilinear estimates from Section 3 along with the counterexamples below certain regularities.
In Section 5, we prove Theorem \ref{THM:LWPonR} by proving the corresponding sharp bilinear estimates in the non-periodic setting.
In Section 6, we prove the ill-posedness results, namely Theorems \ref{THM:illposedonT} and \ref{THM:illposedonR}.
In the appendix, we introduce the vector-valued $X^{s, b}$ spaces, and sketch the proof of Theorem \ref{MTHM:mean0LWP}.
Please see \cite{OHTHESIS} for the full details.

\smallskip

\noindent
{\bf Acknowledgements:} 
The author would like to express his sincere gratitude to his advisor  
Prof. Andrea R. Nahmod. He acknowledges summer support under Prof. Nahmod's  NSF grant DMS 0503542.
He is also grateful to the anonymous refree for his comments.

\section{Notation}

In the periodic setting on $\T_\ld$, the spatial Fourier domain is $\Z/\ld$.
Let $d\xi^\ld$ be the normalized counting measure on $\Z/\ld$, 
and we say $f \in L^p(\Z/\ld)$, $1 \leq p < \infty$ if
\[ \| f \|_{L^p(\mathbb{Z}/\ld)} = \bigg( \int_{\mathbb{Z}/\ld} |f(\xi)|^p d\xi^\ld \bigg)^\frac{1}{p}  
:= \bigg( \frac{1}{2\pi\ld} \sum_{\xi \in \mathbb{Z}/\ld} |f(\xi)|^p \bigg)^\frac{1}{p} < \infty.\]

\noindent
If $ p = \infty$, we have the obvious definition involving the essential supremum.
For $f \in \mathcal{S}(\mathbb{R})$, the Fourier transform of $f$ is defined as
$ \ft{f}(\xi) = \int_{\mathbb{R}} e^{-i x \xi} f(x) dx,$
and its inverse Fourier transform is defined as $\invft{f}(\xi) = \frac{1}{2\pi}\ft{f}(-\xi)$.
If $f \in L^2(\mathbb{T}_\ld)$, 
then the Fourier transform of $f$ is defined as
$ \ft{f}(\xi) = \int_{0}^{2\pi\ld} e^{-i x \xi} f(x) dx$, 
 where $ \xi \in \mathbb{Z}/\ld$, 
and we have the Fourier inversion formula
\[ f(x) = \int_{\mathbb{Z}/\ld} e^{i x \xi} \ft{f}(\xi) d\xi^\ld  = \frac{1}{2\pi\ld} \sum_{\xi \in \mathbb{Z}/\ld} e^{i x \xi} \ft{f}(\xi).\]

\noindent
If the function depends on both $x$ and $t$, we  use ${}^{\wedge_x}$ 
(and ${}^{\wedge_t}$) to denote the spatial (and temporal) Fourier transform, respectively.
However, when there is no confusion, we  simply use ${}^\wedge$ to denote the spatial Fourier transform,
temporal Fourier transform, and  the space-time Fourier transform, depending on the context.

Let $\jb{\, \cdot \,} = 1 + | \cdot |$. 
For $Z= \mathbb{R} $ or $\mathbb{T}_\ld$, 
we define  $X^{s, b}(Z \times \mathbb{R})$ and $X^{s, b}_\al (Z \times \mathbb{R})$
by the norms
\begin{align} 
\|u \|_{X^{s, b}(Z \times \mathbb{R})} & = \| \jb{\xi}^s \jb{\tau - \xi^3}^b \ft{u}(\xi, \tau) \|_{L^2_{\xi, \tau}(Z^\ast \times \mathbb{R})}
\\ 
\|v \|_{X^{s, b}_\al(Z \times \mathbb{R})} & = \| \jb{\xi}^s \jb{\tau - \al \, \xi^3}^b \ft{v}(\xi, \tau) \|_{L^2_{\xi, \tau}(Z^\ast \times \mathbb{R})},
\end{align}

\noindent
where $Z^\ast = \R$ if $Z = \R$ and $Z^\ast = \Z/\ld$ if $Z = \T_\ld$.
Given any time interval $I \subset \mathbb{R}$, we define the local in time $X^{s, b}(Z \times I )$ by
\[ \|u\|_{X_I^{s, b}} = \|u \|_{X^{s, b}(Z \times I )} = \inf \big\{ \|\wt{u} \|_{X^{s, b}(Z \times \mathbb{R})}: {\wt{u}|_I = u}\big\}.\]

\noindent
We define the local in time $X^{s, b}_\al (Z \times I )$ analogously.
Also, in dealing with a product space of two copies of a Banach space $X$, 
we may use $X\times X$ and $X$ interchangeably.

Lastly,
let $\eta \in C^\infty_c(\mathbb{R})$ be a smooth cutoff function supported on $[-2, 2]$ with $\eta \equiv 1$ on $[-1, 1]$.
We use $c,$ $ C$ to denote various constants, usually depending only on $s, b$, and  $\al$.
If a constant depends on other quantities, we  make it explicit.
We use $A\lesssim B$ to denote an estimate of the form $A\leq CB$.
Similarly, we use $A\sim B$ to denote $A\lesssim B$ and $B\lesssim A$
and use $A\ll B$ when there is no general constant $C$ such that $B \leq CA$.
We also use $a+$ (and $a-$) to denote $a + \eps$ (and $a - \eps$), respectively,  
for arbitrarily small $\eps \ll 1$.

\section{Local Well-Posedness on $\T_\ld$, $0< \al <1$}

In this section, we first introduce the spaces for the solutions and the Duhamel terms 
along with known linear estimates and embeddings.
Then, we give the precise statements for the bilinear estimates whose proofs are presented in the next section.
At the end of the section, we present the proof of Theorem \ref{THM:YLWP}.

\subsection{Basic Function Spaces and Linear Estimates}
Let $S(t) = e^{-t\dx^3}$ and $S_\al (t) = e^{-\al t\dx^3}$.
By Duhamel principle, $(u, v)$ is a solution to \eqref{MB} on $[-1, 1]$ if and only if
\begin{equation*}
\begin{cases}
u(t) = \eta(t)S(t) u_0- \eta(t)\int_0^t  S(t- t')  F(t') dt'\\
v(t) = \eta(t)S_\al(t) v_0 - \eta(t) \int_0^t S_\al (t- t') G (t') dt'
\end{cases}
\end{equation*}

\noindent
for $-1 \leq t \leq 1$, where
$F(t') = \eta(t') \dx \big(\frac{v^2}{2}\big) (t')$ and $G(t') = \eta(t') \dx \big(uv\big)(t')$.
The proof of local well-posedness  is based  
on the iteration in the spaces $X^{s, \frac{1}{2}}\times X_{\al}^{s, \frac{1}{2}}$.
However, this space barely fails to be in $C(\R_t; H^s_x \times H^s_x)$.
Thus, we introduce slightly smaller spaces $Y^s$ and $Y_\al^s$ defined via the norm
\begin{align*} 
\|u\|_{Y^{s} } & = \|u\|_{X^{s, \frac{1}{2}}} + \| \jb{\xi}^s  \ft{u}(\xi, \tau) \big\|_{L^2_\xi(\mathbb{Z}/\ld) L^1_\tau( \mathbb{R})} 
\\
 \|v\|_{Y_{\al}^{s}} & = \|v\|_{X_\al^{s,\frac{1}{2}}} + \big\| \jb{\xi}^s  \ft{v}(\xi, \tau) \big\|_{L^2_\xi(\mathbb{Z}/\ld) L^1_\tau ( \mathbb{R})}. 
\end{align*}

\noindent
Then, if $(u, v) \in Y^{s} \times Y_{\al}^{s}$, then $(u, v) \in C(\R_t; H^s_x \times H^s_x)$.
Also, define the spaces $Z^s$ and $Z^s_\al$ via the norm
\begin{align*} 
\|u\|_{Z^{s} } & = \|u\|_{X^{s, -\frac{1}{2}}} + \| \jb{\xi}^s \jb{\tau - \xi^3}^{-1} \ft{u}(\xi, \tau) \|_{L^2_\xi(\mathbb{Z}/\ld) L^1_\tau(\mathbb{R})}
\\ 
 \|v\|_{Z_{\al}^{s}} & = \|v\|_{X_\al^{s, -\frac{1}{2}}} + \| \jb{\xi}^s \jb{\tau - \al \xi^3}^{-1} 
  \ft{v}(\xi, \tau) \|_{L^2_\xi(\mathbb{Z}/\ld) L^1_\tau ( \mathbb{R})} .
\end{align*}

Next, we list known linear estimates and several useful lemmata.
For the proofs, please see \cite{BO1}, \cite{CKSTT4}, \cite{KPV6}, \cite{OHTHESIS}. 

\begin{lemma} Let  $\phi$ be a periodic function on $\mathbb{T}_\ld$.  Then, we have 
\begin{equation}
\| \eta(t)S(t) \phi \|_{X^{s, b}} \lesssim \|\phi\|_{H^s} \text{ and } \ \| \eta(t)S_\al(t) \phi \|_{X_\al^{s, b}} \lesssim \|\phi\|_{H^s}.
\end{equation}

\end{lemma}

\noindent
By the standard computation, the Duhamel terms satisfy

\begin{lemma}

\begin{equation*} \label{duhamelestimate1}
\bigg\| \eta(t)\int_0^t  S(t- t') F(t') dt' \bigg\|_{Y^s} \lesssim \|F\|_{Z^s},  
\text{ and }
\ \bigg\| \eta(t)\int_0^t  S_\al(t- t') G(t') dt' \bigg\|_{Y_\al^s} \lesssim \|G\|_{Z_\al^s}.
\end{equation*}

\end{lemma}

\noindent
Also, we list several embeddings of $X^{s, b}$ and $X_\al^{s, b}$ spaces.

\begin{lemma} \label{embed1}
Let $f(x, t)$ be a function on $ \mathbb{T}_\ld \times \mathbb{R} $.
Then, we have 
\begin{equation*} 
 \| f \|_{L^4_t L^2_x} \lesssim \| f \|_{ X^{0, \frac{1}{4}}} \text{ and } \ \| f \|_{L^4_t L^2_x} \lesssim \| f \|_{X_\al^{0, \frac{1}{4}}} .
\end{equation*}
\end{lemma}

\begin{lemma} \label{embed2}
Let $\ld \geq 1 $ and $\g = \max(C/\ld, 1) $.
Let $f(x, t)$ be a function on $ \mathbb{T}_\ld \times \mathbb{R} $
such that $\supp \ft{f}(\xi, t) \subset [1/\ld, \g] \text{ for all } t \in \mathbb{R}$.
Then, we have
\begin{equation*}
\| f \|_{L^4_t L^\infty_x} \lesssim \ld^{0+} \big\| | \dx |^\frac{1}{2} f \big\|_{X^{0, \frac{1}{4}}} \text{ and } \
\| f \|_{L^4_t L^\infty_x} \lesssim \ld^{0+} \big\| | \dx |^\frac{1}{2} f \big\|_{X_\al^{0, \frac{1}{4}}} .
\end{equation*}

\end{lemma}

\begin{lemma} \label{embed3}
Let $f$ be as in Lemma \ref{embed2}. Then, we have
\begin{equation*} 
\| f \|_{L^2_t L^\infty_x} \lesssim \ld^{0+} \big\| | \dx |^\frac{1}{2} f \big\|_{ L^2_{x, t}} .
\end{equation*}
\end{lemma}

\noindent
Moreover, we have the $L^4$ Strichartz estimate due to Bourgain \cite{BO1}. 
\begin{lemma} \label{L^4strichartz}
Let $\ld \geq 1$.
Let $f(x, t)$ be a function on $ \mathbb{T}_\ld \times \mathbb{R} $. 
Then, we have 
\begin{equation*} 
 \| f \|_{L^4_{x, t}} \lesssim \| f \|_{ X^{0, \frac{1}{3}}} \text{ and } \ \| f \|_{L^4_{x, t}} \lesssim \| f \|_{X_\al^{0, \frac{1}{3}}},
\end{equation*}
where the implicit constants $C(\ld)$ and $ C_\al(\ld)$ are decreasing functions of $\ld$.
In particular, we have $C(\ld) \leq C(1)$ and $C_\al(\ld) \leq C_\al(1)$  for $\ld \geq 1$.

\end{lemma}

\subsection{Bilinear Estimates and Local Well-Posedness on $\T_\ld$}
In this subsection, we make precise statements of the sharp bilinear estimates on $\T_\ld \times \R$,
where the constants $C_1$ and $C_2$ are expressed in term of the spatial period $\ld$.
Then, we briefly discuss how the proof of Theorem \ref{THM:YLWP} follows from the standard contraction argument  in \cite{CKSTT4},
pointing out the difference due to the constants $C_1(\ld)$ and  $C_2(\ld)$.

\begin{proposition} \label{PROP:bilinear1}
The bilinear estimate  
\begin{align} 
 \label{Zbilinear1}
\| \dx ( v_1v_2)  \|_{Z^{s} (\mathbb{T_\ld} \times \mathbb{R} )} & \lesssim C_1(\ld)
\|v_1\|_{Y_\al^{s}(\mathbb{T_\ld} \times \mathbb{R})}\|v_2\|_{Y_\al^{s}(\mathbb{T_\ld} \times \mathbb{R})} 
\end{align}

\noindent
holds for $s \geq \min ( 1, \frac{1}{2} + \frac{1}{2} \nu_{c_1} + ) $, 
where 
\begin{equation} \label{YC1}
 C_1(\ld) = \begin{cases}
\ld^{\frac{1}{2}+ \frac{1}{2}\nu_{c_1}+}, & \text{for } 0 \leq \nu_{c_1}  < 1 \\
\ld^{0+}, & \text{for } \nu_{c_1}  \geq 1. 
\end{cases} 
\end{equation}

\end{proposition}

\begin{proposition} \label{PROP:bilinear2} 
Assume the mean 0 condition on $u$. 
Then, the bilinear estimate
\begin{align}  \label{Zbilinear2}
\| \dx ( u v)  \|_{Z_\al^{s} (\mathbb{T_\ld} \times \mathbb{R} )} 
 \lesssim C_2(\ld)
\|u\|_{Y^{s}(\mathbb{T_\ld} \times \mathbb{R})}\|v\|_{Y_\al^{s}(\mathbb{T_\ld} \times \mathbb{R})} 
\end{align}

\noindent
holds for $s \geq \min ( 1, \frac{1}{2} + \frac{1}{2} \max( \nu_{d_1} ,  \nu_{d_2})+ ) $, where
\begin{equation} \label{YC2}
 C_2(\ld) = \begin{cases}
\ld^{\frac{1}{2}+ \frac{1}{2} \max( \nu_{d_1} ,  \nu_{d_2})+}, & 0 \leq \max( \nu_{d_1} , \nu_{d_2})  < 1 \\
\ld^{0+}, &  \max( \nu_{d_1} ,  \nu_{d_2})  \geq 1. 
\end{cases} 
\end{equation}

\end{proposition}

\noindent
Moreover, we  also establish the sharpness of these estimates.
\begin{proposition} \label{PROP:counterexample1}
If $s < \min ( 1, \frac{1}{2} + \frac{1}{2} \nu_{c_1}  )$, then
the bilinear estimate
\begin{equation}
\label{scaledbilinear1}
\| \dx ( v_1 v_2)  \|_{X^{s, b-1} (\mathbb{T_\ld} \times \mathbb{R} )} 
 \lesssim 
\|v_1\|_{X_\al^{s, b}(\mathbb{T_\ld} \times \mathbb{R})}\|v_2\|_{X_\al^{s, b}(\mathbb{T_\ld} \times \mathbb{R})}
\end{equation}

\noindent
fails for any $b \in \mathbb{R}$. 
Hence, \eqref{Zbilinear1} fails for $s < \min ( 1, \frac{1}{2} + \frac{1}{2} \nu_{c_1}  )$.
\end{proposition}

\begin{proposition}\label{PROP:counterexample2}
\textup{(a)} Without the mean 0 condition on $u$, 
the bilinear estimate
\begin{equation}
\label{scaledbilinear2}
\| \dx ( uv )  \|_{X_\al^{s, b-1} (\mathbb{T_\ld} \times \mathbb{R} )} 
 \lesssim 
\|u\|_{X^{s, b}(\mathbb{T_\ld} \times \mathbb{R})}\|v\|_{X_\al^{s, b}(\mathbb{T_\ld} \times \mathbb{R})}
\end{equation}
\textup{(}and hence \eqref{Zbilinear2}\textup{)}
fails for any $s, b \in \mathbb{R}$.

\noindent
\textup{(b)}
If $s < \min ( 1, \frac{1}{2} + \frac{1}{2} \max( \nu_{d_1} , \nu_{d_2})  )$, then
  \eqref{scaledbilinear2}
\textup{(}and hence \eqref{Zbilinear2}\textup{)}
fails for any $b \in \mathbb{R}$. 
\end{proposition}

In particular, for almost every $\al \in (0, 1)$, 
the bilinear estimates \eqref{Zbilinear1} and \eqref{Zbilinear2} hold for $s > \frac{1}{2}$.  
See Remark \ref{JJREMARK2}.
Also, see Remark \ref{JJREMARK3} for the endpoint cases $ s = \frac{1}{2} + \frac{1}{2} \nu_{c_1}$ when $\nu_{c_1} <1$
and  $ s = \frac{1}{2} + \frac{1}{2}\max( \nu_{d_1} , \nu_{d_2})$ when $\max( \nu_{d_1} , \nu_{d_2}) <1$.

\begin{proof}[Proof of Theorem \ref{THM:YLWP}]
Fix $\ld \geq 1$ and let $s_0 = \frac{1}{2} + \frac{1}{2}\max( \nu_{c_1},\nu_{d_1},\nu_{d_2})$ and $s^\ast = \max(1, s_0+)$.
Also, fix $s \geq s^\ast$. Note that
for $s \geq 1$,  \eqref{Zbilinear1} and \eqref{Zbilinear2}
hold with $\max(C_1(\ld), C_2(\ld)) = \ld^{0+}$.
When $\max( \nu_{c_1}, \nu_{d_1},  \nu_{d_2}) <1$,   \eqref{Zbilinear1} and \eqref{Zbilinear2}
hold for $s \in (s_0, 1)$ with  $\max(C_1(\ld), C_2(\ld)) = \ld^{s_0+} < \ld^\frac{3}{2}$, since $s_0 < 1$.
i.e. $\max(C_1(\ld), C_2(\ld))$ in either case is less than $\ld^\frac{3}{2}$ which is the gain of the power of $\ld$ from scaling.

Then, by replacing $\ld^{0+}$ with $\max(C_1(\ld), C_2(\ld))$ in Subsection 7.5 of \cite{CKSTT4}, 
we see that \eqref{MB} is well-posed in $H^s(\T_\ld)$, $s \geq s^\ast$, on a time interval of size $\sim 1$,
assuming that $\| (u_0, v_0) \|_{H^s}$ is sufficiently small.
i.e. $ \max(C_1(\ld), C_2(\ld)) \| (u_0, v_0) \|_{H^s(\T_\ld)} \ll 1$.

Now, let $(u_0, v_0) \in H^s(\T_\ld) \times H^s(\T_\ld)$.
Then, \eqref{MB} is well-posed in $H^s(\T_\ld)$ on a small time interval $[0, \dl]$ 
if and only if the $\s$-rescaled problem is well-posed in $H^s(\T_{\s\ld})$ on the time interval $[0, \s^3 \dl ]$.
Note that with $\beta = 0$ if $s_0 \geq 1$ and $\beta = s_0$ if $s_0 < 1$, we have
\[\max(C_1(\s\ld), C_2(\s\ld)) \|(u_0^\s, v_0^\s)\|_{H^s(\T_{\s\ld})} \leq  (\s\ld)^{\beta+} \s^{-\frac{3}{2}} \|(u_0, v_0)\|_{H^s(\T_{\ld})}
\ll 1\]

\noindent
for $s \geq s^\ast$, if $\s = \s(\ld_, \| (u_0, v_0) \|_{H^s(\T_\ld)})$ is taken to be sufficiently large.
The rest follows from the argument in \cite{CKSTT4}.
\end{proof}

\begin{remark} \label{JJREMARK5} \rm 
The constant $C(\ld)$ of the bilinear estimate \eqref{KPVZbilinear} for KdV is $\ld^{0+}$ 
and this does not cause any trouble in scaling for both local and global arguments (c.f. \cite{CKSTT4}.)
Unlike the KdV case, the LWP of \eqref{MB} without smallness assumption of the initial data crucially depends 
on the fact that  $\max(C_1(\ld), C_2(\ld)) $ 
can be controlled by $\ld^\frac{3}{2}$.
Moreover, it is essential to minimize $C_1(\ld)$ and $C_2(\ld)$
for the application of the I-method for $0 < \al < 1$ in establishing the global well-posedness of \eqref{MB}.
See \cite{OH2}.
\end{remark}

\begin{remark} \label{JJREMARK2}  \rm
From Remark \ref{JJDIO}, we see that the minimal type index $\nu_{c_1} = 0 $ for almost every $c_1 \in \mathbb{R}.$
This, in turn, implies that $\nu_{c_1} = 0$ for almost every $\al \in (0, 1)$,
 since $c_1$ is explicitly defined in terms of $\al$ as in  \eqref{c_1}.
Similarly, we have $\nu_{d_1}= \nu_{d_2} = 0$ for almost every $\al \in (0, 1)$.
Hence, the local well-posedness of \eqref{MB} in $H^s(\T) \times H^s(\T)$
for $s > \frac{1}{2}$ for almost every $\al \in (0, 1)$ follows as claimed in Theorem \ref{THM:YLWP}.

\end{remark}

\section{Diophantine Conditions in the Bilinear Estimates on $\T_\ld$}

In this section, we  present the proofs of the main bilinear estimates \eqref{Zbilinear1} and \eqref{Zbilinear2}
along with the counterexamples showing their sharpness.
For the conciseness of the presentation, we focus on \eqref{Zbilinear1}.
We  first present the construction of the counterexamples below the  regularities specified 
in Propositions \ref{PROP:counterexample1} 
since it shows the particular resonance interactions more clearly.
The counterexamples are constructed on $\T$ in the same manner as in \cite{KPV4}.
Then, we prove \eqref{Zbilinear1}  on $\T_\ld$ for $\ld \geq 1$.
We  need to go through an 
intricate argument for $\ld \geq 1$ in order to minimize the constants $C_1(\ld)$.
It follows from the proofs that the estimates hold for $s \geq 0$ as long as the functions are restricted to the domains away from the resonance sets.

Recall the numbers $c_1, c_2, d_1,$ and $ d_2$ from the resonance equations 
\eqref{JJreseq1} and \eqref{resonance2}.
i.e.  $c_1$ and $c_2$ solve
\begin{equation} \label{resonance11}
1 - \al c_1^3 - \al c_2^3 = 0
\end{equation}

\noindent
with $c_1 + c_2 = 1$, and
 $d_j$  solves
$\al - d_j^3 - \al \wt{d_j} \vphantom{|}^3 =0$ for $ j = 1, 2$, 
where $\wt{d_j} = 1 - d_j$.

Given $x \in \mathbb{R}$, let $[ x ] = $ the closest integer to $x$,
satisfying 
$[c_1 N ] + [c_2N] = c_1 N+ c_2 N = N$ and 
$[d_j N ] + [\wt{d_j} N ] = d_j N + \wt{d_j} N = N$.
(Note that $[\, \cdot \, ]$ is not the usual integer-part function.)
Now, let $\theta_N = [ c_1 N ]- c_1N  = c_2N  - [ c_2 N ]$.
Then, from the definition of the minimal type index $\nu_{c_1}$, 
for any $\eps > 0$, there exist infinitely many $N \in \mathbb{N}$ such that
\begin{equation} \label{theta_N111}
|\theta_N |= \min_{k \in \mathbb{Z}} | c_1 N - k | = N \min_{k \in \mathbb{Z}} \bigg| c_1  - \frac{k}{N} \bigg| < \frac{1}{N^{1+\nu_{c_1}-\eps}}.
\end{equation}

\noindent
Then, given any $\eps>0$, there are infinitely many $N$ such that
\begin{equation} \label{resres11}
\big|\al [c_1N ]^3  + \al [c_2N ]^3 -N^3\big| 
= | 3 \al (c_1 - c_2) N^2\theta_N + 3\al N \theta_N^2 | \lesssim N^{1-\nu_{c_1} +\eps} 
\end{equation}

\noindent for sufficiently large $N$ satisfying \eqref{theta_N111}.
On the other hand, from the definition of the minimal type index $\nu_{c_1}$, we have, for any $\eps >0$, 
\begin{equation} \label{resres111}
\big|\al [c_1N ]^3  + \al [c_2N ]^3 -N^3\big| \gtrsim N^{1-\nu_{c_1} -\eps} 
\end{equation}

\noindent
for all sufficiently large $N \in \mathbb{Z}$.
In a similar manner, we see that given any $\eps>0$, there are infinitely many $N$ such that
\begin{equation} \label{resres22}
\big| [d_j N ]^3  + \al [\wt{d_j}N ]^3 -\al N^3\big| \lesssim N^{1-\nu_{d_j} +\eps},
j = 1, 2.
\end{equation}

\noindent 
On the other hand, for any $\eps >0$, we have
\begin{equation} \label{resres222}
\big| [d_j N ]^3  + \al [\wt{d_j}N ]^3 -\al N^3\big| \gtrsim N^{1-\nu_{d_j} -\eps} 
\end{equation}

\noindent
for all sufficiently large $N \in \mathbb{Z}$.
Note that \eqref{resres11} and \eqref{resres22} are used to construct the counterexamples,
whereas
\eqref{resres111}
and \eqref{resres222} play crucial roles in proving \eqref{Zbilinear1} and \eqref{Zbilinear2}.

\begin{proof}[Proof of Proposition \ref{PROP:counterexample1}]

We  construct counterexamples to the bilinear estimate \eqref{scaledbilinear1}
for $s < \min ( 1, \frac{1}{2} + \frac{1}{2} \nu_{c_1}  )$.
First, define the bilinear operator $\mathcal{B}_{s, b} $ by
\begin{equation} \label{Zdual1}
 \mathcal{B}_{s, b} (f, g) (\xi, \tau) =  \frac{\xi \jb{\xi}^s }{\jb{\tau - \xi^3}^{1-b}}
 \intt_{\substack{\xi_1+ \xi_2 = \xi \\
\tau_1 + \tau_2 = \tau}} \frac{f(\xi_1, \tau_1) g(\xi_2, \tau_2)}{\jb{\xi_1}^s 
\jb{\xi_2}^s \jb{\tau_1 - \al \xi_1^3}^b  \jb{\tau_2 - \al \xi_2^3}^b} d\xi_1^\ld d\tau_1.
\end{equation}

\noindent
For simplicity, let $\ld = 1$. 
Then, 
(\ref{scaledbilinear1}) holds
if and only if
\begin{equation} \label{dualbilinear11}
\left\| \mathcal{B}_{s,b}(f, g) \right\|_{L^2_{\xi, \tau}} \lesssim \| f \|_{L^2_{\xi_1, \tau_1}} \| g \|_{L^2_{\xi_2, \tau_2}} . 
\end{equation}

\noindent
\textbf{$\bullet$ Case (1):} $c_1, c_2 \in \mathbb{Q}$ 

Then, $\nu_{c_1} = \infty$, i.e. $ \min ( 1, \frac{1}{2} + \frac{1}{2} \nu_{c_1}  ) = 1$.
Since $c_1, c_2 \in \mathbb{Q}$, there are infinitely many $N \in \mathbb{N}$ such that $c_1N, c_2N \in \mathbb{Z}$.
For such $N$, define $f_N$ and $g_N$ by 
$f_N(\xi, \tau) = a_\xi \chi_{1} (\tau - \al\xi^3)$ and 
$g_N(\xi, \tau) = b_\xi \chi_{2} (\tau - \al \xi^3)$, 
where $\chi_{\theta}(\cdot)$ is the characteristic function of the interval $[-\theta, \theta]$ and 
$a_\xi = 1$ if and only if $\xi = c_1 N$ and $b_\xi = 1$ if and only if $\xi = c_2 N$.
Then, $ \| f_N \|_{L^2_{\xi, \tau}} \sim \|g_N \|_{L^2_{\xi, \tau}} \sim 1 $ for all such $N$.
(In the following, we  always choose $f_N$ and $g_N$ so that their $L^2_{\xi, \tau}$ norms are of size $\sim$ 1 for all $N \in \mathbb{N}$.)
Now, let $ A_N = \{(\xi, \tau) \in \mathbb{Z} \times \mathbb{R} : \xi = N, | \tau - N^3 | \leq 1\} $.
 Note that
$|\tau_1 - \al (c_1 N)^3 | \leq 1$ and $|\tau - N^3 | \leq 1$ implies
$|\tau_2 - \al (c_2 N)^3 | \leq 2$,
since $\tau = \tau_1 + \tau_2$ and $ N^3 = \al (c_1 N)^3 + \al (c_2 N)^3 $.
Then, we have
$ \| \mathcal{B}_{s, b} (f_N, g_N) \|_{L^2_{\xi, \tau}(\mathbb{Z} \times \mathbb{R})} 
 \geq \| \mathcal{B}_{s, b} (f_N, g_N) \|_{L^2_{\xi, \tau}(A_N)} 
 \sim N^{1-s} .$
\noindent
Hence, if \eqref{dualbilinear11} holds, then, 
we must have $N^{1-s} \lesssim 1$ for all $N \in \mathbb{N}$ with $c_1 N, c_2 N \in \mathbb{N}$,
where the implicit constant is independent of $N$.
Therefore,  \eqref{scaledbilinear1} fails for  $s < 1$ regardless of the value of $b$.

\noindent
\textbf{$\bullet$ Case (2):} $c_1, c_2 \in \mathbb{R} \setminus \mathbb{Q}$

First, we  show $ 1/3 \leq b \leq 2/3$ if \eqref{dualbilinear11} holds.
For $N \in \mathbb{N}$, define 
$f_N(\xi, \tau) = a_\xi \chi_{1} (\tau - \al \xi^3)$ 
and $g_N(\xi, \tau) = b_\xi \chi_{2} (\tau - \al \xi^3)$, 
where
$a_\xi = 1$ if and only if $\xi = 1$ and $b_\xi = 1$ if and only if $\xi = N-1$.
Then, on the support of $f_N(\xi_1, \tau_1)$ and $g_N(\xi_2, \tau_2)$, we have 
\begin{align*}
| \tau - N^3 | & \leq |\tau_1 - \al \cdot 1^3  | + |\tau_2 - \al (N-1)^3 | 
+ | \al + \al (N-1)^3 - N^3| \\
& \leq 2 + |(\al - 1) N^3 - 3 \al N^2 + 3 \al N |
\lesssim N^3
\end{align*}
for sufficiently large $N \in \mathbb{N}$. 
Similarly, we have $ | \tau - N^3 | \gtrsim N^3$ for large $N$.
Thus, we have $| \tau - N^3 | \sim N^3$ for all large $N$.
Now, let $ A_N = \{ (\xi, \tau) \in \mathbb{Z} \times \mathbb{R} : \xi = N,  | \tau - \al  - \al (N - 1)^3 | \leq 1 \}$.
Note that $g_N(N-1, \tau_2) \equiv 1 $ on $\supp ( f_N) \cap A_N$.
Then, we have
$\| \mathcal{B}_{s, b} (f_N, g_N) \|_{L^2_{\xi, \tau}(\mathbb{Z} \times \mathbb{R})} 
 \geq \| \mathcal{B}_{s, b} (f_N, g_N) \|_{L^2_{\xi, \tau}(A_N)} 
 \gtrsim N^{-2 + 3b}$
for all sufficiently large $N$.
Hence, if \eqref{dualbilinear11} holds for any $s \in \mathbb{R}$, 
we must have $N^{-2 + 3b} \lesssim 1$ for all large $N$, i.e.  
$b \leq 2/3$.

Note that if \eqref{dualbilinear11} holds, by duality, we see that
a similar inequality must hold for $ \wt{\mathcal{B}}_{s, b} (\cdot, \cdot)$, where
\[ \wt{\mathcal{B}}_{s, b} (f, g) (\xi_1, \tau_1) = \frac{1}{\jb{\xi_1}^s\jb{\tau_1 - \al \xi_1^3}^b}
\frac{1}{2\pi}
\sum_{\xi = \xi_1 + \xi_2} \intt_{\tau = \tau_1 + \tau_2} \frac{ \xi \jb{\xi}^s f(\xi, \tau) g(\xi_2, \tau_2)}
{\jb{\xi_2}^s \jb{\tau - \xi^3}^{1- b} \jb{\tau_2 - \al \xi_2^3}^b} d\tau. \]
i.e. \eqref{dualbilinear11} is equivalent to 
\begin{equation} \label{dualbilinear12}
\big\| \wt{\mathcal{B}}_{s,b}(f, g) \big\|_{L^2_{\xi_1, \tau_1}} \lesssim \| f \|_{L^2_{\xi, \tau}} \| g \|_{L^2_{\xi_2, \tau_2}} . 
\end{equation}

Now, let
$
f_N(\xi, \tau) = a_\xi \chi_{1} (\tau - \xi^3)$ and $g_N(\xi, \tau) = b_\xi \chi_{2} (\tau - \al \xi^3)$, 
where
$a_\xi = 1$ if and only if $\xi = N$ and $b_\xi = 1$ if and only if $\xi = N-1$.
Then, on the support of $f_N(\xi, \tau)$ and $g_N(\xi_2, \tau_2)$, we have 
\begin{align*}
| \tau_1 - \al  | & \leq |\tau - N^3  | + |\tau_2 - \al (N-1)^3 | 
+ |  \al+ \al (N-1)^3 - N^3| \\
& \leq 2 + |(\al - 1) N^3 - 3 \al N^2 + 3 \al N |
\lesssim N^3
\end{align*}
for sufficiently large $N \in \mathbb{N}$. 
Similarly, we have $ | \tau - \al | \gtrsim N^3$ for large $N$.
Thus, we have $| \tau_1 - \al | \sim N^3$.
Now, let $ \wt{A}_N = \{ (\xi_1, \tau_1) \in \mathbb{Z} \times \mathbb{R} : \xi_1 = 1, 
 | \tau_1 - N^3  + \al (N - 1)^3 | \leq 1 \}$.
Note that  $g_N(\xi_2, \tau_2) \equiv 1 $ on $\supp ( f_N) \cap \wt{A}_N$.
Then, 
we have
$\| \wt{\mathcal{B}}_{s, b} (f_N, g_N) \|_{L^2_{\xi_1, \tau_1}(\mathbb{Z} \times \mathbb{R})} 
 \geq \| \wt{\mathcal{B}}_{s, b} (f_N, g_N) \|_{L^2_{\xi_1, \tau_1}(\wt{A}_N)} 
 \gtrsim N^{1 - 3b}$ for all sufficiently large $N$.
Hence, if (\ref{dualbilinear12}) holds for any $s \in \mathbb{R}$, 
we have $N^{1 - 3b} \lesssim 1$ for all large $N$, i.e. $b \geq 1/3$. 
Therefore, if \eqref{dualbilinear12} holds for some $s \in \mathbb{R}$, we have $1/3 \leq b \leq 2/3$.

 Now, we  show that $s \geq \min(1, \frac{1}{2} + \frac{1}{2}\nu_{c_1})$ whenever  \eqref{scaledbilinear1} holds. 
Let
$
f_N(\xi, \tau) = a_\xi \chi_{1} (\tau - \al \xi^3) $
and $g_N(\xi, \tau) = b_\xi \chi_{2} (\tau - \al \xi^3), $
where
$a_\xi = 1$ if and only if $ \xi = [c_1 N ] $
and  $b_\xi = 1$ if and only if $ \xi = [ c_2 N ] $.
Let $N$ be an integer such that \eqref{theta_N111} holds.
\noindent On the support of $f_N(\xi_1, \tau_1)$ and $g_N(\xi_2, \tau_2)$, we have
\begin{align} \label{taucurve11}
|\tau - N^3| & \leq \big|\tau_1 - \al [c_1 N ]^3 \big| + \big|\tau_2 - \al [c_2 N ]^3 \big|
+ \big| \al [c_1 N ]^3 + \al [c_2 N  ]^3 - N^3 \big| \notag \\
& \leq 3 + | 3 \al (c_1 - c_2) N^2\theta_N + 3\al N \theta_N^2 | 
\lesssim N^2 \theta_N < N^{1-\nu_{c_1}+\eps}, 
\end{align}

\noindent
for sufficiently large $N$ satisfying \eqref{theta_N111}.
i.e. $\jb{\tau - N^3} \lesssim \max( N^{1-\nu_{c_1}+\eps}, 1)$. 
Now, let $ A_N = \{ (\xi, \tau) \in \mathbb{Z} \times \mathbb{R} : \xi = N, 
 \big| \tau - \al [c_1 N ]^3 - \al [c_2 N ]^3 \big| \leq 1 \}$.
Note that  $g_N(\xi_2, \tau_2) \equiv 1 $ on $\supp ( f_N) \cap A_N$.
Then, 
we have
$ \mathcal{B}_{s, b}  (f_N, g_N) (N, \tau) 
 \gtrsim 
 N^{1-s}\max( N^{1-\nu_{c_1}+\eps}, 1)^{b-1}$ 
on $A_N$.

If $\nu_{c_1} > 1$, then we have $\nu_{c_1} - \eps > 1$ for sufficiently small $\eps > 0$.
i.e. $\max( N^{1-\nu_{c_1}+\eps}, 1) = 1$.
Then, we have $\| \mathcal{B}_{s, b} (f_N, g_N) \|_{L^2_{\xi, \tau}(\mathbb{Z} \times \mathbb{R})} 
\gtrsim N^{1-s } .$
Hence, if (\ref{dualbilinear11}) holds, then we must have $ s \geq 1 = \min(1, \frac{1}{2} + \frac{1}{2}\nu_{c_1})$.
If $\nu_{c_1} \leq 1$, then $\nu_{c_1} - \eps < 1$. 
i.e. $\max( N^{1-\nu_{c_1}+\eps}, 1) =  N^{1-\nu_{c_1}+\eps}$,
and thus 
$\| \mathcal{B}_{s, b} (f_N, g_N) \|_{L^2_{\xi, \tau}(\mathbb{Z} \times \mathbb{R})} 
\gtrsim N^{-s+ \nu_{c_1} - \eps -(1-\nu_{c_1}+\eps)b } .$
Hence, if (\ref{dualbilinear11}) holds, 
we must have $s - \nu_{c_1} + \eps \geq  (1-\nu_{c_1}+\eps) b$.

Now, let
$f_N(\xi, \tau) = a_\xi \chi_{1} (\tau - \xi^3) $
and $g_N(\xi, \tau) = b_\xi \chi_{2} (\tau - \al \xi^3)$, 
where
$a_\xi = 1$ if and only if $ \xi = N$
and $ b_\xi = 1$ if and only if $\xi = [ c_2 N ]$.
Then, on the support of $f_N(\xi, \tau)$ and $g_N(\xi_2, \tau_2)$, we have
\begin{align} \label{tau1curve11}
\big|\tau_1 - \al [c_1N ]^3\big| & \leq |\tau - N^3 | + \big|\tau_2 - \al [c_2 N ]^3 \big|
+ \big| \al [c_1 N ]^3 + \al [c_2 N ]^3 - N^3 \big| \notag \\
& \leq 3 + | 3 \al (c_1 - c_2) N^2\theta_N + 3\al N \theta_N^2 | 
\lesssim N^{1-\nu_{c_1}+\eps}
\end{align}

\noindent
for sufficiently large $N$ satisfying \eqref{theta_N111}.
i.e. $\jb{\tau_1 - \al [c_1N]^3} \lesssim \max( N^{1-\nu_{c_1}+\eps}, 1) = N^{1-\nu_{c_1}+\eps}$. 
Now, let $ \wt{A}_N = \{ (\xi_1, \tau_1) \in \mathbb{Z} \times \mathbb{R} : \xi_1 = [c_1N ], 
 \big| \tau_1 - N^3 + \al [c_2 N ]^3 \big| \leq 1 \}$.
Note that  $g_N(\xi_2, \tau_2) \equiv 1 $ on $\supp ( f_N) \cap \wt{A}_N$.
Then, we have
$\| \wt{\mathcal{B}}_{s, b} (f_N, g_N) \|_{L^2_{\xi_1, \tau_1}(\mathbb{Z} \times \mathbb{R})} 
 \geq \| \wt{\mathcal{B}}_{s, b} (f_N, g_N) \|_{L^2_{\xi_1, \tau_1}(\wt{A}_N)} 
 \gtrsim N^{1 - s - (1-\nu_{c_1}+\eps)b}.$
Thus, if \eqref{dualbilinear12} holds, 
we must have $(1-\nu_{c_1}+\eps)b \geq 1-s$.

Putting two results together, we have $ 1-s \leq (1-\nu_{c_1}+\eps)b \leq s - \nu_{c_1} + \eps$.
i.e. we must have $ s \geq \frac{1}{2} + \frac{1}{2} (\nu_{c_1}- \eps) \text{ for any } \eps >0 $.
Therefore, we must have $ s \geq \min( 1, \frac{1}{2} + \frac{1}{2} \nu_{c_1}) $ if  \eqref{scaledbilinear1} holds.
\end{proof}

Now, we present the proof of Proposition \ref{PROP:bilinear1}.

\begin{proof}
[Proof of Proposition \ref{PROP:bilinear1}]

Fix $\ld \geq 1$.  We  prove  the bilinear estimate \eqref{Zbilinear1} in two steps.

\noindent
$\bullet$ {\bf Part 1:} First, we  prove \eqref{scaledbilinear1}. 
Let  $\mathcal{B}_{s} (\cdot, \cdot) = \mathcal{B}_{s, \frac{1}{2}} (\cdot, \cdot)$
where  $\mathcal{B}_{s, b} (\cdot, \cdot)$ is defined in \eqref{Zdual1}.
Then,  \eqref{scaledbilinear1} holds if and only if
\begin{equation} \label{c_1dual11} 
\left\| \mathcal{B}_{s}(f, g) \right\|_{L^2 (d \xi^\ld d\tau) }
\lesssim C_1(\ld) \| f \|_{L^2 (d \xi^\ld d\tau) } \| g \|_{L^2 (d \xi^\ld d\tau) } . 
\end{equation}

Now, let $\G_\xi(\xi_1) = - \xi^3 +\al \xi_1^3+\al \xi_2^3 = 3 \al \xi \xi_1^2 - 3 \al \xi^2 \xi_1 - (1 - \al) \xi^3$.
For fixed $\xi$, $\G_\xi(\xi_1)$ is a quadratic function in $\xi_1$ such that $\G_\xi(c_1\xi) = \G_\xi(c_2\xi) = 0$.
For now, assume $\xi > 0$. Then,  $c_1 \xi > 0$ is a positive root of a convex parabola $\G_\xi(\xi_1)$,
and we have
\begin{align*} \textstyle
\G_\xi(c_1 \xi - \tfrac{1}{\ld}) = -3\al \xi \frac{1}{\ld}\Big(\frac{\sqrt{-3 + 12 \al^{-1}}}{3}\xi - \frac{1}{\ld} \Big)
\sim - \frac{1}{\ld} \xi^2, 
\end{align*}
for $\xi \geq L_\al : = \frac{6}{\sqrt{-3 + 12\al^{-1}}} \geq \frac{6}{\ld \sqrt{-3 + 12\al^{-1}}}$.  
Hence, for $| \xi_1 - c_1 \xi | \geq  1/\ld$ and $|\xi| \geq L_\al$, we have
$ |\G_\xi (\xi_1)| \gtrsim \xi^2/\ld.$
The same computation with $c_2$ shows that we have, for $|\xi| \geq L_\al$,  
\begin{equation} \label{away11}
\MAX := \max \big(  \jb{\tau - \xi^3},    \jb{\tau_1 - \al \xi_1^3},   \jb{\tau_2 - \al \xi_2^3} \big) 
\gtrsim |\G_\xi (\xi_1)| \gtrsim \frac{\xi^2}{\ld} \text{ on } A ,
\end{equation}
where 
$ A = \big\{ (\xi, \xi_1, \xi_2) \in (\mathbb{Z}/\ld)^3: 
\xi = \xi_1 + \xi_2, \  | \xi_1 - c_1 \xi | \geq  1/\ld \text{ and }| \xi_1 - c_2 \xi | \geq  1/\ld \big\}. $

Now, we'd like to obtain a lower bound for $\G_\xi(\xi_1)$ near the resonances, i.e. on $A^c$. 
Recall that \eqref{resres111} holds for $\xi, \xi_1, \xi_2 \in \mathbb{Z}$.
Since $\xi, \xi_1, \xi_2 \in \mathbb{Z}/\ld$,
we need to compute its dependece on $\ld$ explicitly.
Assume $| \xi_1 - c_1 \xi | < 1/\ld$.
Let $\dl = c_1 \xi - \xi_1$. Then, $|\dl| < 1/\ld$.
Moreover, from the definition of the minimal type index $\nu_{c_1}$, we have, for any $\eps > 0$, 
\[ |\dl| = |\xi_1 - c_1 \xi| = |\xi| \bigg| c_1 - \frac{\xi_1}{\xi}\bigg| = |\xi| \bigg| c_1 - \frac{\ld \xi_1}{\ld \xi}\bigg|
\geq |\xi| \frac{K_\eps}{(\ld |\xi|)^{2+\nu_{c_1}+\eps}} = \frac{K_\eps |\xi|^{-1-\nu_{c_1}-\eps}}{\ld^{2+\nu_{c_1}+\eps}},\]
since $\ld\xi, \, \ld\xi_1 \in \mathbb{Z}$.
Therefore, for $|\xi| \geq L_\al$, we have
\begin{align} \label{near11}
\MAX 
& \gtrsim |\G_\xi(\xi_1)|  
  \geq  3\al  L_\al \xi^2 |\dl|  
\sim \frac{|\xi|^{1-\nu_{c_1}-\eps}}{\ld^{2+\nu_{c_1}+\eps}}.
\end{align}

Now, we  prove \eqref{c_1dual11}.
Without loss of generality, assume $f$ and $g$ are nonnegative.
By symmetry, we also assume $ \jb{\tau_1-\al \xi_1^3} \geq \jb{\tau_2-\al \xi_2^3}$.

\noindent
$\bullet$ {\bf Case (1):}  $|\xi| \leq L_\al$ 

For $s \geq 0$, we have
$\frac{|\xi| \jb{\xi}^s }{\jb{\xi_1}^s\jb{\xi_2}^s} \lesssim \frac{ \jb{\xi}^s }{\jb{\xi_1}^s\jb{\xi_2}^s} \lesssim 1$.
Also, $\jb{\tau - \xi^3} \geq 1$.
Now, let $\ft{F}(\xi_1, \tau_1) = \jb{\tau_1 - \al \xi_1^3}^{-\frac{1}{2}} f(\xi_1, \tau_1)$
and $\ft{G}(\xi_2, \tau_2) = \jb{\tau_2 - \al \xi_2^3}^{-\frac{1}{2}} g(\xi_2, \tau_2)$.
Then, by H\"older inequality and Lemma \ref{L^4strichartz}, we have
\begin{align*} 
\| \mathcal{B}_{s}  (f, g) \|_{L^2_{\xi, \tau} }  & = \| FG\|_{L^2_{x, t}} 
 \lesssim \|F\|_{X_\al^{0, \frac{1}{3}}} \|G\|_{X_\al^{0, \frac{1}{3}}} 
\leq  \| f \|_{L^2_{\xi, \tau} } \| g \|_{L^2_{\xi, \tau} } . 
\end{align*}

\noindent
$\bullet$
{\bf Case (2):}  $| \xi_1 - c_1 \xi | \geq  1/\ld$ and $| \xi_1 - c_2 \xi | \geq  1/\ld$ with $ |\xi| > L_\al$ 

In this case, we have \eqref{away11}:
$ \MAX =\max \big(  \jb{\tau - \xi^3},    \jb{\tau_1 - \al \xi_1^3},   \jb{\tau_2 - \al \xi_2^3} \big) 
 \gtrsim \xi^2/\ld.$

\noindent
$\circ$ Subcase (2.a): $\jb{\tau - \xi^3} \gtrsim \xi^2/\ld$ 

In this case, we have $\frac{|\xi|\jb{\xi}^s} {\jb{\xi_1}^s\jb{\xi_2}^s}\frac{1}{\jb{\tau - \xi^3}^\frac{1}{2}} 
\lesssim \frac{|\xi|}{\jb{\tau - \xi^3}^\frac{1}{2}}\lesssim \ld^\frac{1}{2}$
for $s \geq 0$.
Then, by following the computation in Case (1), it follows that \eqref{c_1dual11} holds  for $s \geq 0$,  giving $C_1(\ld) = \ld^\frac{1}{2}$.

\noindent
$\circ$ Subcase (2.b): $\jb{\tau_1 - \al \xi_1^3} \gtrsim \xi^2/\ld$  

Note that by duality, \eqref{c_1dual11} is equivalent to 
\begin{align} \label{duality11} 
\bigg| \iint \mathcal{B}_{s}  (f, g) (\xi, \tau) h(\xi, \tau)  d \xi^\ld d\tau\bigg|
 \lesssim  C_1(\ld) \| f \|_{L^2 (d \xi^\ld d\tau) } \| g \|_{L^2 (d \xi^\ld d\tau) } \| h \|_{L^2 (d \xi^\ld d\tau) }. 
\end{align}

\noindent
In this case, we have $\frac{|\xi|\jb{\xi}^s} {\jb{\xi_1}^s\jb{\xi_2}^s}\frac{1}{\jb{\tau_1 - \al \xi_1^3}^\frac{1}{2}} 
\lesssim \frac{|\xi|}{\jb{\tau_1 - \al \xi_1^3}^\frac{1}{2}}\lesssim \ld^\frac{1}{2}$
for $s \geq 0$.
Let $\ft{G}(\xi_2, \tau_2) = \jb{\tau_2 - \al \xi_2^3}^{-\frac{1}{2}} g(\xi_2, \tau_2)$
and $\ft{H}(\xi, \tau) = \jb{\tau -  \xi^3}^{-\frac{1}{2}} h(-\xi, -\tau)$.
Then, by H\"older and Lemma \ref{L^4strichartz}, we have, for $s \geq 0$, 
\begin{align*}
 \text{LHS of }  \eqref{duality11} &\lesssim \ld^\frac{1}{2} \bigg| \iint \invft{f}(x, t) G(x, t) H(x, t) dx dt \bigg|  
\leq \ld^\frac{1}{2} \| f\|_{L^2_{\xi, \tau} } \|G\|_{L^4_{x, t}} \|H\|_{L^4_{x, t}} \\
& \lesssim \ld^\frac{1}{2} \| f\|_{L^2_{\xi, \tau} } \|G\|_{X_\al^{0\frac{1}{3}}} \|H\|_{X^{0\frac{1}{3}}}
\lesssim \ld^\frac{1}{2} \| f \|_{L^2_{\xi, \tau} } \| g \|_{L^2_{\xi, \tau} } \| h \|_{L^2_{\xi, \tau}  }. 
\end{align*}

\noindent
$\bullet$
{\bf Case (3):}  $| \xi_1 - c_1 \xi | <  1/\ld$ or $| \xi_1 - c_2 \xi | <  1/\ld$ with $ |\xi| > L_\al$ 

Without loss of generality, assume $|\xi_1 - c_1 \xi|< 1/\ld$.
Then, $|\xi_2 - c_2\xi| = |\xi_1 - c_1 \xi|< 1/\ld$.
Thus, we have
$ \jb{\xi_1} \sim \jb{\xi_2} \sim \jb{\xi} $.  Hence, for $s \geq 1$, we have
\[ \frac{|\xi| \jb{\xi}^s}{\jb{\xi_1}^s\jb{\xi_2}^s} = \frac{|\xi| \jb{\xi}}{\jb{\xi_1}\jb{\xi_2}} \frac{ \jb{\xi}^{1-s}}{\jb{\xi_1}^{1-s}\jb{\xi_2}^{1-s}} \lesssim 1.\]  
Then, the rest follows from Case (1), giving $C_1(\ld) = 1$

This proves  \eqref{scaledbilinear1} for $s \geq 1$ with $C_1(\ld) = \ld^\frac{1}{2}$, 
which corresponds to the case when $\nu_{c_1} \geq 1$.

\begin{remark} \rm
By repeating the same computation in Case (2) with
$| \xi_1 - c_1 \xi | \geq 1$ and $| \xi_1 - c_2 \xi | \geq  1$ for $ |\xi| > L_\al$
and Case (3) with
$| \xi_1 - c_1 \xi | <  1$ or $| \xi_1 - c_2 \xi | <  1$ for $ |\xi| > L_\al$, 
we can indeed obtain \eqref{scaledbilinear1} for $s \geq 1$ with $C_1(\ld) = 1$, which is sharp if $\nu_{c_1} \geq 1$.

\end{remark}

Now, we  consider the case when $\nu_{c_1} < 1$.
Fix $ s \in (\frac{1}{2} + \frac{1}{2}\nu_{c_1}, 1)$.
Then, there exists small $\eps > 0$ such that $ s \geq \frac{1}{2} + \frac{1}{2}(\nu_{c_1}+\eps) $.
In this case, we have \eqref{near11}:
$ \MAX  \gtrsim \frac{|\xi|^{1-\nu_{c_1}-\eps}}{\ld^{2+\nu_{c_1}+\eps}}.$
Since $ \jb{\xi_1} \sim \jb{\xi_2} \sim \jb{\xi} $ in this case, we have
\begin{align} \label{near111}
 \frac{|\xi| \jb{\xi}^s}{\jb{\xi_1}^s\jb{\xi_2}^s} \frac{1}{\MAX^\frac{1}{2}} 
& \lesssim \ld^{1+\frac{1}{2}(\nu_{c_1}+\eps) } \, \text{ for } \, s \geq \tfrac{1}{2} + \tfrac{1}{2}(\nu_{c_1}+\eps) .
\end{align}

Note that we do {\it not} have an infinite sum in $\xi_1$ in this case, 
since $| \xi_1 - c_1 \xi | <  1/\ld$ or $| \xi_1 - c_2 \xi | <  1/\ld$, i.e.
for each fixed $\xi$, there are at most 4 values of $\xi_1$ in the sum.  
For the rest of the argument, we assume that $\xi_1 = c_1 \xi + \dl_\xi$ with $0 \leq \dl_\xi < 1/\ld$.
(The cases for $0 \geq \dl_\xi > - 1/\ld$ and $| \xi_1 - c_2 \xi | <  1/\ld$ are exactly the same.)
Then, we have
\begin{align} \label{nosum11}  
\mathcal{B}_{s} & (f, g) (\xi, \tau) \notag \\
&=  \frac{\xi \jb{\xi}^s }{\jb{\tau - \xi^3}^\frac{1}{2}} \frac{1}{2\pi \ld}
 \intt_{   \tau_1 + \tau_2 = \tau } 
\frac{f(c_1\xi + \dl_\xi , \tau_1) g(c_2\xi - \dl_\xi, \tau_2)}{\jb{\xi_1}^s 
\jb{\xi_2}^s \jb{\tau_1 - \al (c_1\xi + \dl_\xi)^3}^\frac{1}{2}  \jb{\tau_2 - \al (c_2\xi - \dl_\xi)^3}^\frac{1}{2}}  d\tau_1.
\end{align}

\noindent
$\circ$ Subcase (3.a): $\MAX = \jb{\tau - \xi^3} $  

Let $\ft{F}(\xi_1, \tau_1) = \jb{\tau_1 - \al \xi_1^3}^{-\frac{1}{2}} f(\xi_1, \tau_1)$
and $\ft{G}(\xi_2, \tau_2) = \jb{\tau_2 - \al \xi_2^3}^{-\frac{1}{2}} g(\xi_2, \tau_2)$.
Then, from \eqref{near111} and \eqref{nosum11}, we have
\begin{align*}
\| \mathcal{B}_{s}  (f, g) \|_{L^2(d \xi^\ld d\tau)  }  
&\lesssim 
\ld^{\frac{1}{2}(\nu_{c_1}+\eps) }   
\Big\| \big\| \ft{F}(c_1\xi + \dl_\xi, \cdot) *_\tau \ft{G}(c_2\xi - \dl_\xi, \cdot) \big\|_{L^2_\tau} \Big\|_{L^2 (d \xi^\ld ) } \\
& = \ld^{\frac{1}{2}(\nu_{c_1}+\eps) }   
\Big\| \big\| \ft{F}^{{}^x}(c_1\xi + \dl_\xi, t)  \ft{G}^{{}^x}(c_2\xi - \dl_\xi, t) \big\|_{L^2_t} \Big\|_{L^2 (d \xi^\ld ) }.
\end{align*}

\noindent By changing the order of the integrations, 
\begin{align*}
 &\leq \ld^{\frac{1}{2}(\nu_{c_1}+\eps) }   
\Big\| \big\| \ft{F}^{{}^x}(\cdot, t) \big\|_{L^\infty (d \xi^\ld ) } \big\| \ft{G}^{{}^x}(\cdot, t) \big\|_{L^2 (d \xi^\ld ) } \Big\|_{L^2_t } \\
& \leq \ld^{\frac{1}{2}(\nu_{c_1}+\eps) }   
\big\| \| F(\cdot, t) \|_{L^1_x  } \| G(\cdot, t) \|_{L^2_x } \big\|_{L^2_t }
\leq \ld^{\frac{1}{2}+ \frac{1}{2}(\nu_{c_1}+\eps) }   
\big\| \| F(\cdot, t) \|_{L^2_x  } \| G(\cdot, t) \|_{L^2_x } \big\|_{L^2_t }, 
\end{align*}
where the last inequality follows from H\"older inequality on $F$ in $x$.
By $L^4_t, L^4_t$ H\"older inequality and Lemma \ref{embed1}, we have
\[ \| \mathcal{B}_{s}  (f, g) \|_{L^2_{\xi, \tau} }  
 \lesssim \ld^{\frac{1}{2}+ \frac{1}{2}(\nu_{c_1}+\eps) }   
\|  F(\cdot, t) \|_{X_\al^{0, \frac{1}4} } \| G(\cdot, t) \|_{X_\al^{0, \frac{1}4} } 
 \leq \ld^{\frac{1}{2}+ \frac{1}{2}(\nu_{c_1}+\eps) }   
\| f \|_{L^2_{\xi, \tau}  } \| g \|_{L^2_{\xi, \tau}}. \]

\noindent
$\circ$ Subcase (3.b): $\MAX = \jb{\tau_1 - \al \xi_1^3} $

By duality, it is enough to show \eqref{duality11}
as in Subcase (2.b).
Using \eqref{near111} and \eqref{nosum11} followed by
$L^4_t, L^4_t$ H\"older inequality and Lemma \ref{embed1}, 
we can show \eqref{duality11} with $C_1(\ld) = \ld^{\frac{1}{2}+\frac{1}{2}(\nu_{c_1}+\eps) } $
in a similar manner as in Subcase (3.a),
and thus we omit the details.
This completes the proof of  \eqref{scaledbilinear1}.

\noindent
$\bullet$ {\bf Part 2:} Now, consider the $L^2_\xi(\mathbb{Z}/\ld)L^1_\tau(\mathbb{R})$ part of the $Z^s$ norm.
The basic idea is to reduce the proof to Part 1 either by Cauchy-Schwarz or Lemma \ref{closetocubic11} below (c.f. \cite{CKSTT4}.)
Note that it suffices to prove
\begin{align} \label{Z^s11}
 \| \jb{\tau - \xi^3}^{-\frac{1}{2}} B_{s} (f, g) (\xi, \tau) \|_{L^2(d \xi^\ld) L^1_\tau} 
\lesssim  \ C_1(\ld) \| f \|_{L^2(d \xi^\ld d\tau )} \| g \|_{L^2(d \xi^\ld d\tau)}
\end{align}

\noindent
where $B_{s}(\cdot, \cdot)$ is as in Part 1.
By Cauchy-Schwarz inequality in $\tau$, 
\[ \text{LHS of } \eqref{Z^s11} \leq  \Big\| \| \jb{\tau - \xi}^{-\frac{1}{2}-} \|_{L^2_\tau}
\|\mathcal{Q}_s(\xi, \tau) \|_{L^2_\tau} \Big\|_{L^2(d \xi^\ld)}
\lesssim \|\mathcal{Q}_s(\xi, \tau) \|_{L^2(d \xi^\ld d\tau)},
\]
where
$\mathcal{Q}_s(\xi, \tau) = 
\jb{\tau - \xi^3}^{0+} B_{s} (f, g) (\xi, \tau)$.

Note that \eqref{Z^s11} basically follows from Part 1,
if $|\xi| \leq L_\al$, $s \geq 1$, or
$\MAX = \jb{\tau_1 - \al \xi_1^3}$ or $\jb{\tau_2 - \al \xi_2^3}$.
Hence, we  assume $\MAX = \jb{\tau - \xi^3}$
and $s < 1$, i.e. $\nu_{c_1} < 1$.

\noindent
$\bullet$
{\bf Case (4):}  $\jb{\tau_1 - \al \xi_1^3} \gtrsim \jb{\tau - \xi^3}^\frac{1}{100}$
or $\jb{\tau_2 - \al \xi_2^3} \gtrsim \jb{\tau - \xi^3}^\frac{1}{100}$ 

Without loss of generality, assume $\jb{\tau_1 - \al \xi_1^3} \gtrsim \jb{\tau -  \xi^3}^\frac{1}{100}$.
Then, we have 
\[\mathcal{Q}_s(\xi, \tau) \lesssim 
\frac{\xi \jb{\xi}^s }{\jb{\tau - \xi^3}^{\frac{1}{2}}}
 \iintt_{ \substack{  \xi_1+ \xi_2 = \xi \\  \tau_1 + \tau_2 = \tau }} 
\frac{f(\xi_1, \tau_1) g(\xi_2, \tau_2)}{\jb{\xi_1}^s 
\jb{\xi_2}^s \jb{\tau_1 - \al \xi_1^3}^\frac{1}{3}  \jb{\tau_2 - \al \xi_2^3}^\frac{1}{2}} d \xi_1^\ld d\tau_1.
\]

\noindent
Then, again,  \eqref{Z^s11} basically follows from Subcases (2.a) or (3.a).

\noindent
$\bullet$
{\bf Case (5):}  $\jb{\tau_1 - \al \xi_1^3}, \jb{\tau_2 - \al \xi_2^3} \ll \jb{\tau - \xi^3}^\frac{1}{100}$ 

Recall that $\G_\xi(\xi_1) =  (\tau - \xi^3) - (\tau_1- \al \xi_1^3) -(\tau_2 - \al \xi_2^3)$.
Thus, in this case, we have 
\[ \tau - \xi^3 = \G_\xi(\xi_1) +  o\big(\jb{\tau- \xi^3}^\frac{1}{100}\big) = \G_\xi(\xi_1) +  o\big(|\G_\xi(\xi_1)|^\frac{1}{100}\big). \]

\noindent
Let
$ \Omega(\xi) = \big\{ \eta \in \mathbb{R} : \eta =  \G_\xi(\xi_1) +  o\big(|\G_\xi(\xi_1)|^\frac{1}{100}\big)$
 for some $ \xi_1 \in \mathbb{Z}/\ld \big\}$. 
Then, we have the following lemma (c.f. \cite[Lemma 7.4]{CKSTT4}) whose proof is postponed.

\begin{lemma} \label{closetocubic11}
Let $|\xi| > L_\al \gtrsim 1$.  Then, for all dyadic $M \geq 1$, we have
\begin{equation} \label{omega11}
\big| \Omega(\xi) \cap \{ |\eta|\sim M \} \big| \lesssim \ld M^\frac{2}{3}.
\end{equation}

\end{lemma}

\noindent
 Then, using this lemma, we have
$ \int \jb{\tau - \xi^3}^{-1}  \chi_{\Omega(\xi)}(\tau - \xi^3) d\tau 
\lesssim \ld^{0+}$
as in \cite[p.737]{CKSTT4}.
Then, by Cauchy-Schwarz in $\tau$, we have
\begin{align*}
\text{LHS of } \eqref{Z^s11} 
& \leq  
 \big\| \| \jb{\tau - \xi^3}^{-\frac{1}{2}} \chi_{\Omega(\xi)}(\tau - \xi^3) \|_{L^2_\tau}
\| \mathcal{B}_s (f, g) \|_{L^2_\tau} \big\|_{L^2 (d\xi^\ld)} \\
& \lesssim \ld^{0+} \| \mathcal{B}_s (f, g)  \|_{L^2 (d\xi^\ld d \tau)}, 
\end{align*}
and thus the proof is reduced to Subcases (2.a) or (3.a),  
establishing \eqref{Z^s11} for $s \geq {\frac{1}{2}+\frac{1}{2}(\nu_{c_1}+\eps)  }$
with $C_1(\ld) = \ld^{\frac{1}{2}+\frac{1}{2}(\nu_{c_1}+\eps) + } $.
\end{proof}

\begin{proof} [Proof of Lemma \ref{closetocubic11}]
Without loss of generality, assume $\xi$ is positive.
 Since $\G_\xi(\xi_1) = -\xi^3 + \al \xi_1^3 + \al \xi_2^3$ is symmetric in $\xi_1$ and $\xi_2$, assume $|\xi_1| \geq |\xi_2|$.
Then, $\xi_1 \geq \xi/2$ since $\xi = \xi_1 + \xi_2$.
Once we fix $\xi_1$, we have
\[\big| \big\{ \eta \in \mathbb{R} : |\eta| \sim M, \  \eta =  \G_\xi(\xi_1) +  o\big(|\G_\xi(\xi_1)|^\frac{1}{100}\big) \big\} \big| \sim M^\frac{1}{100}.\] 

\noindent
Now, we need to estimate the number of possible values of $\xi_1 \in \mathbb{Z}/\ld$
such that 
\begin{equation} \label{simM11}
\big|\G_\xi(\xi_1) +  o\big(|\G_\xi(\xi_1)|^\frac{1}{100}\big) \big|\sim M. 
\end{equation}

\noindent
For the following argument, let  $|\eta| \sim M$ and $\xi \sim N$, dyadic.

\noindent
$\bullet$
{\bf Case (1):}  $\xi/2 \leq \xi_1 \leq c_1 \xi - 1$  

First note that
$\G_\xi(\xi) =  -(1 - \al) \xi^3 < 0$ and $\G_\xi(\frac{\xi}{2}) =  -\frac{3}{2}\al \xi^2 < 0$.
Since $\G_\xi(\xi_1)$ is increasing in this range of $\xi_1$, we have $|\G_\xi(\xi_1)| \sim N^p$ for some $p\in[2, 3]$
Thus, we have $M \sim N^p$, i.e. $N \sim M^\frac{1}{p}$.
Since $ \xi_1 \in [\xi/2, c_1 \xi - 1]$ and $\xi \sim N \sim M^\frac{1}{p}$, there are $\sim \ld M^\frac{1}{p}$ many possible values of $\xi_1$ in this case.
Hence, the contribution to \eqref{omega11} is at most $\sim  \ld M^\frac{1}{p} M^\frac{1}{100} \leq \ld M^\frac{2}{3}$.

 \noindent
$\bullet$
{\bf Case (2):}  $|\xi_1 - c_1\xi | \leq 1$ 

Then, there are $\sim \ld$ many possible values of $\xi_1$ in this case.
Hence, the contribution to \eqref{omega11} is at most $\sim  \ld M^\frac{1}{100}$.

\noindent
$\bullet$
{\bf Case (3):}  $\xi_1 \geq c_1 \xi+1$ 

In this case, $\G_\xi(\xi_1)$ is positive and thus we have $\G_\xi(\xi_1) \sim M$.
Note that 
$\G_\xi(\xi_1) 
= 3 \al \xi (\xi_1 - \tfrac{\xi}{2})^2 - (1 + \tfrac{3}{4}\al)\xi^3 $.
Now, let $\wt{\G}_\xi(\xi_1) = 3 \al \xi (\xi_1 - \frac{\xi}{2})^2.$
i.e. for $\xi$ fixed,  $\wt{\G}_\xi(\xi_1)$ is a upward parallel translate of the parabola $\G_\xi(\xi_1)$.
Then, we have
\begin{align*}
  \#\big\{ \xi_1 \in \mathbb{Z}/\ld : \xi_1 \geq c_1 &  \xi + 1 \text{ and } \G_\xi(\xi_1) \sim M \big\} \\
& \leq \#\big\{ \xi_1 \in \mathbb{Z}/\ld  : \xi_1 \geq c_1 \xi + 1 \text{ and } \wt{\G}_\xi(\xi_1) \sim M \big\},
\end{align*}

\noindent
since the graph of $\G_\xi(\xi_1)$ is steeper than that of $\wt{\G}_\xi(\xi_1)$ for any fixed range of height $\sim M \geq 1$.
Furthermore, we have
\[\#\big\{ \xi_1 \in \mathbb{Z}/\ld : \xi_1 \geq c_1 \xi + 1 \text{ and } \wt{\G}_\xi(\xi_1) \sim M \big\} 
\leq \#\big\{ \xi_1 \in \mathbb{Z}/\ld : 0 \leq \wt{\G}_\xi(\xi_1) \leq 2M \big\}\]
and the latter can be at most $\sim \ld (M/N)^\frac{1}{2} \lesssim \ld M^\frac{1}{2}$ since $N \sim |\xi| \gtrsim 1$.
Hence, the contribution to \eqref{omega11} is at most $\sim  \ld M^\frac{1}{2} M^\frac{1}{100} \leq \ld M^\frac{2}{3}$.
\end{proof}

The proofs of
Propositions \ref{PROP:bilinear2}
and  \ref{PROP:counterexample2}
for the second bilinear estimate \eqref{Zbilinear2}
are analogous to those of Propositions \ref{PROP:bilinear1}
and  \ref{PROP:counterexample1},
(using \eqref{resres22} and \eqref{resres222}
instead of \eqref{resres11} and \eqref{resres111}.)
In the following, we only show how 
\eqref{scaledbilinear2}
(and hence \eqref{Zbilinear2}) fails for any $s, b \in \mathbb{R}$
without the mean 0 assumption $u$,
and omit the rest.
See \cite{OHTHESIS} for other details.

\begin{proof}[Proof of Proposition \ref{PROP:counterexample2} \textup{(}a\textup{)}]
For simplicity, let $\ld = 1$. We  construct counterexamples to \eqref{scaledbilinear2}. 
Define the bilinear operator $\mathcal{B}_{s, b} $ by
\[ \mathcal{B}_{s, b} (f, g) (\xi, \tau) =  \frac{\xi \jb{\xi}^s }{\jb{\tau - \al \xi^3}^{1-b}}
\sum_{\xi_1+ \xi_2 = \xi} \intt_{\tau_1 + \tau_2 = \tau} \frac{f(\xi_1, \tau_1) g(\xi_2, \tau_2)}{\jb{\xi_1}^s 
\jb{\xi_2}^s \jb{\tau_1 - \xi_1^3}^b  \jb{\tau_2 - \al \xi_2^3}^b} d\tau_1.
\]

\noindent
Then, 
(\ref{scaledbilinear2}) holds
if and only if
\begin{equation} \label{dualbilinear22}
\left\| \mathcal{B}_{s,b}(f, g) \right\|_{L^2_{\xi, \tau}} \lesssim \| f \|_{L^2_{\xi_1, \tau_1}} \| g \|_{L^2_{\xi_2, \tau_2}} . 
\end{equation}

\noindent
Let $
f_N(\xi, \tau) = a_\xi \chi_{1} (\tau - \xi^3) $ and 
$g_N(\xi, \tau) = b_\xi \chi_{2} (\tau - \al \xi^3)$, 
where
$a_\xi = 1$ if and only if $ \xi = 0$
and $ \ b_\xi = 1$ if and only if $ \xi =  N$.
Now, let $ A_N = \{(\xi, \tau) \in \mathbb{Z} \times \mathbb{R} : \xi = N, | \tau - \al N^3 | \leq 1\} $.
 Note that
$|\tau_1 - 0^3 | \leq 1$ and $ |\tau - \al N^3 | \leq 1$ implies 
$|\tau_2 - \al  N^3 | \leq 2. $
Then, on $A_N$, we have
$\mathcal{B}_{s, b} (f_N, g_N) (N, \tau) 
 \sim   \frac{N^{1+s} }{N^s}  \int \chi_1(\tau_1 ) d\tau_1
\sim  N$,
and thus
$\| \mathcal{B}_{s, b} (f_N, g_N) \|_{L^2_{\xi, \tau}(\mathbb{Z} \times \mathbb{R})} 
 \gtrsim N$.
Hence, if \eqref{dualbilinear22} holds, then, 
we must have $N \lesssim 1$ for all $N \in \mathbb{N}$,
which is impossible.
Therefore,  \eqref{scaledbilinear2} (and hence \eqref{Zbilinear2}) can not hold for any  $s, b \in \mathbb{R}$. 
\end{proof}

We conclude this section by stating several important remarks.

\begin{remark} \label{JJREMARK1} \rm
 Unlike KdV (i.e. $\al = 1$), we do {\it not} need the mean 0 condition on the functions in Proposition \ref{PROP:bilinear1}.  
More importantly, while the resonance for $\al = 1$ makes the bilinear estimate \eqref{Zbilinear1}
fail without the mean 0 condition regardless of regularity,
 the resonance for $0 < \al < 1$ can be treated by {\it assuming a higher regularity.}
This fact can be explained as follows. 
For example, the resonance for $\al = 1$ on $\mathbb{T}$ occurs when $\xi_1 = \xi$ and $\xi_2 = 0$ for all $\xi \in \mathbb{Z}$.
In some sense, we can say that 
there are infinitely many resonances accumulated at $\xi_2 = 0$.
This makes it impossible to treat the resonance at $\xi_2 = 0$
and thus \eqref{Zbilinear1} for $\al = 1$ fails for all $s \in \mathbb{R}$ without the mean 0 condition.

Now, consider the case when $ 0 < \al < 1$. 
Recall that, in dealing with \eqref{Zbilinear1}, 
we have resonances at $\xi_1 = c_1 \xi$ and $\xi_2 = c_2 \xi$ for each $\xi \in \mathbb{Z}$.
Also, note that  $ \lim_{\al \to 1} c_1 = 1$ and $\lim_{\al \to 1} c_2 = 0$.
Then, as soon as $\al \ne 1$, these infinitely many resonances accumulated at $\xi_2 = 0$ for $\al = 1$
are suddenly distributed over infinitely many distinct $\xi_2 = c_2 \xi$ for $ \xi \in \mathbb{Z}$,
which lets us treat each resonance by assuming a higher regularity.
This is what makes the $\al = 1$ case very different from $\al \in (0, 1)$.
\end{remark}

\begin{remark} \label{JJREMARK3}  \rm
When $0 \leq \nu_{c_1} < 1$, then the bilinear estimate \eqref{Zbilinear1} holds for $s> \frac{1}{2} + \frac{1}{2} \nu_{c_1} $ 
and fails for $s < \frac{1}{2} + \frac{1}{2} \nu_{c_1} $.  
i.e. we are missing the endpoint $s = \frac{1}{2} + \frac{1}{2} \nu_{c_1} $.
Actually, the proof of Proposition \ref{PROP:bilinear1} can be applied to prove 
 \eqref{Zbilinear1}  for $s = \frac{1}{2} + \frac{1}{2} \nu_{c_1} $ 
if $c_1$ is {\it  of its own minimal type}. i.e. if there exists $K = K(c_1) > 0$ such that 
for all pairs of integers $(m, n)$, we have
\begin{equation} \label{endpointc_1} 
\left| c_1 - \frac{m}{n} \right| \geq \frac{K}{ |n|^{2+\nu_{c_1}}} .
\end{equation}

\noindent
Moreover, if \eqref{endpointc_1} does not hold, then  \eqref{Zbilinear1} fails for $s = \frac{1}{2} + \frac{1}{2} \nu_{c_1} $.
Showing \eqref{endpointc_1} for given $c_1 \in \R$ is a  problem of number-theoretic nature, 
and we do not pursue this issue here.
In Bambusi-Paleari \cite{BP} and Berti-Bolle \cite{BB}, it is observed that the set
$A = \{ c_1 \in \R : \nu_{c_1} = 0$ and \eqref{endpointc_1} holds$\}$
is uncountable and of measure 0.
This implies that when $\nu_{c_1} = 0$, the bilinear estimate \eqref{Zbilinear1} holds for the endpoint $s = \frac{1}{2}$ 
only for $c_1$ in  this Cantor-like set $A$ of measure 0.
The same remark applies to the bilinear estimate \eqref{Zbilinear2} in Proposition \ref{PROP:bilinear2}.
\end{remark}

\begin{remark} \label{JJREMARK4}  \rm
In Proposition \ref{PROP:bilinear2}, 
we assume the mean 0 condition on $u$.  
This is due to the infinite accumulation of resonance at $\xi_1 = 0$.
(c.f. Remark \ref{JJREMARK1}.)
If the mean of $u_0$ is not 0, then we can consider $(U, v)$ instead of $(u, v)$, 
where $U(x, t)  = u(x, t) - (2\pi \ld)^{-1}\ft{u_0}(0)$.
Then, one needs to consider 
\begin{equation} \label{JJMB}
\begin{cases}
U_t + U_{xxx} + v v_x = 0 \\
v_t + \al v_{xxx} + (2\pi \ld)^{-1}\ft{u_0}(0) v_x + (Uv)_x = 0.
\end{cases}
\end{equation}

\noindent
Note that this just adds a harmless first order linear term $v_x$ in the second equation.
  See \cite{BO1}.
Also, note that the linear parts of \eqref{JJMB} are not mixed unlike \eqref{mean0MB} in the appendix.

\end{remark}

\begin{remark} \label{JJREMARK6} \rm 
$c_1, d_1,$ and $  d_2$ in defined \eqref{c_1} and \eqref{d_1 and d_2}
are real numbers if and only if $\al \in (0, 1) \cup (1, 4]$.
In particular, the same argument can be applied to establish the local well-posedness for $\al \in (0, 1) \cup (1, 4]$
whose regularity depends on $\nu_{c_1}, \nu_{d_1}$, and $\nu_{d_2}$.
When $\al = 4$, \eqref{MB} is LWP on $\mathbb{T}_\ld$ only for $s\geq 1$ since $c_1 = \frac{1}{2} \in \mathbb{Q}$.
On $\mathbb{R}$, the proof of LWP seems to break down when $\al = 4$. See Remark \ref{KPVremark1}.

When $\al <0$ or $\al > 4$, the resonance equations \eqref{JJreseq1} and \eqref{resonance2}
do not have a real solution for any $\xi \in \mathbb{R}$.
Let's consider \eqref{Zbilinear1} in Proposition \ref{PROP:bilinear1}.
From the resonance equation \eqref{JJreseq1}, we have
\begin{align*} 
|\xi^3 - \al \xi_1^3 -  \al \xi_2^3|
= |\xi| \big|3\al \big(\xi_1 -\tfrac{\xi}{2}\big)^2 + (-1+ \tfrac{\al}{4})\xi^2\big|
\gtrsim \max(|\xi_1 \xi_2\xi|, |\xi|^3).
\end{align*}

\noindent
Thus, we gain $\tfrac{3}{2}$ derivatives 
and the bilinear estimate \eqref{Zbilinear1} holds for $s \geq -\frac{1}{2}$
(with the mean 0 condition) as in the KdV and $\al = 1$ case.
The same remark applies to the second bilinear estimate \eqref{scaledbilinear2}
and the bilinear estimates \eqref{Rbilinear1} and \eqref{Rbilinear2} in the non-periodic setting.
For example, see \cite{AC} for $\al = -1$ on $\mathbb{R}$.
\end{remark}

\section{Well-Posedness on $\R$, $0<\al<1$}

In this section, we  establish the sharp  well-posedness in $L^2(\R) \times L^2(\R)$ for $0 < \al < 1$
by proving the bilinear estimates \eqref{Rbilinear1} and \eqref{Rbilinear2}.
Let $S(t) = e^{-t \dx^3} $ and $S_\al(t) = e^{-\al t \dx^3} $.
By standard computation \cite{KPV6}, we have

\begin{lemma} Let  $\phi$ be a function on $\mathbb{R}$.  Then, we have 
\begin{equation} \label{freepartR}
\| \eta(t)S(t) \phi \|_{X^{s, b}} \lesssim \|\phi\|_{H^s} \text{ and } \| \eta(t)S_\al(t) \phi \|_{X^{s, b}} \lesssim \|\phi\|_{H^s}.
\end{equation}

\end{lemma}

\begin{lemma}
Let $F$, $G$ be functions on $\R \times \R$. Then we have 
\begin{equation*} 
\|\eta  (S*_R F)\|_{X^{s, b}}  \lesssim \|F\|_{X^{s, b-1}}
\text{ and }\|\eta (S_\al*_R G)\|_{X_\al^{s, b}}  \lesssim \|G\|_{X_\al^{s, b-1}}
\end{equation*}

\noindent
where $*_R$ denotes the retarded convolution, i.e.
$S*_R F(t) =\int_0^t  S(t- t') F(t') dt'$.

\end{lemma}

Then, Theorem \ref{THM:LWPonR} follows once we prove  \eqref{Rbilinear1} and \eqref{Rbilinear2}.
Indeed, the following propositions show that the estimates 
are sharp in $L_x^2(\R)$.

\begin{proposition} \label{PROP:Rbilinear} 
The bilinear estimates \eqref{Rbilinear1} and  \eqref{Rbilinear2}
hold for $s \geq 0$ with some $b > \frac{1}{2}$. 
\end{proposition}

\begin{proposition} \label{PROP:Rcounterexample} 
If $s < 0$, the bilinear estimates \eqref{Rbilinear1} and  \eqref{Rbilinear2}
fail  for any $b \in \mathbb{R} $. 
\end{proposition}

\noindent
The proof of Proposition \ref{PROP:Rbilinear} is based on the usual argument 
with H\"older inequality, calculus lemmata, and change of variables in integration. (c.f. \cite{KPV4}.)
Since $\al \ne1$, we need to separate the domain more carefully.
First, we list some calculus lemmata.

\begin{lemma} [ Kenig-Ponce-Vega \cite{KPV4},  Bekiranov-Ogawa-Ponce \cite{BOP1}]\label{calc1}
For $l > \frac{1}{2}$, we have

\noindent
\textup{(a)}
\[ \int_\mathbb{R} \frac{dx}{\jb{x - \al}^{2l}\jb{x - \beta}^{2l}} \lesssim \frac{1}{\jb{ \al - \beta}^{2l}}.\]

\noindent
\textup{(b)}
\[ \int_\mathbb{R} \frac{dx}{\jb{x }^{2l} \sqrt{| \al - x |}} \lesssim \frac{1}{\jb{ \al}^{\frac{1}{2}}}.\]

\noindent
\textup{(c)}
For $l >\frac{1}{3}$,
\[ \int_\mathbb{R} \frac{dx}{\jb{ x^3  + a_2 x^2  + a_1x+a_0}^l} \lesssim  1. \]

\end{lemma}

\begin{proof}[Proof of Proposition \ref{PROP:Rbilinear}]

Note that $\jb{\xi}^s \lesssim \jb{\xi_1}^s\jb{\xi_2}^s$ for $s \geq 0$, 
and thus we prove \eqref{Rbilinear1} and \eqref{Rbilinear2} only for $ s= 0$.
First, we  prove  \eqref{Rbilinear1}.
As usual,  define the bilinear operator $\mathcal{B}_{s, b}$ by 
\begin{equation} \label{Rdual}
\mathcal{B}_{s, b} \big(f, g\big) (\xi, \tau) = \frac{\xi \jb{\xi}^s}{\jb{\tau - \xi^3}^{1-b} }
\iintt_{\substack{\xi = \xi_1 + \xi_2 \\\tau = \tau_1+\tau_2}}
\frac{f(\xi_1, \tau_1) g(\xi_2, \tau_2)}{ \jb{\xi_1}^s \jb{\xi_2}^s \jb{\tau_1 - \al \xi_1^3}^b \jb{\tau_2-\al\xi_2^3}^b}d\xi_1d\tau_1.
\end{equation}

\noindent
Then, \eqref{Rbilinear1} holds for $ s= 0$ if and only if
\begin{equation} \label{Rdual11} 
\left\| \mathcal{B}_{0, b}(f, g) \right\|_{L^2_{\xi, \tau} }
\lesssim  \| f \|_{L^2_{\xi, \tau} } \| g \|_{L^2_{\xi, \tau}} . 
\end{equation}

\noindent
We have the following lemma, which is basically proved in \cite[Lemma 2.4]{KPV4}.

\begin{lemma} \label{calc2}
For $\frac{1}{2} < b \leq \frac{3}{4}$, we have
\begin{equation} \label{calc22}
\sup_{\xi, \tau} \frac{|\xi|}{\jb{\tau - \xi^3}^{1-b} }
\bigg( \iint \frac{d \xi_1 d\tau_1}{\jb{\tau_1 - \al \xi_1^3}^{2b}\jb{\tau_2 - \al \xi_2^3}^{2b}} \bigg)^\frac{1}{2} \lesssim 1.
\end{equation}
\end{lemma}

\noindent
Then, by H\"older inequality and Lemma \ref{calc2}, we have
\begin{align*}
 \left\| \mathcal{B}_{0, b}(f, g) \right\|_{L^2_{\xi, \tau} }
&\lesssim \bigg\|\frac{|\xi|}{\jb{\tau - \xi^3}^{1-b} }
\bigg( \iint \frac{d \xi_1 d\tau_1}{\jb{\tau_1 - \al \xi_1^3}^{2b}\jb{\tau_2 - \al \xi_2^3}^{2b}} \bigg)^\frac{1}{2} \bigg\|_{L^\infty_{\xi, \tau}} \\
& \hphantom{XXXXXX} \times \bigg\| \bigg( \iint |f(\xi_1, \tau_1)|^2 |g(\xi - \xi_1, \tau - \tau_1)|^2  d \xi_1 d \tau_1\bigg)^\frac{1}{2} \bigg\|_{L^2_{\xi, \tau}} \\
& \lesssim \| f \|_{L^2_{\xi, \tau} } \| g \|_{L^2_{\xi, \tau}} . 
\end{align*}

\noindent
Hence,  \eqref{Rdual11} holds for $s\geq 0$ and $\frac{1}{2} < b \leq \frac{3}{4}$. 

\begin{remark} \label{KPVremark1}
\rm
By going through the proof of Lemma \ref{calc2} in \cite[Lemma 2.4]{KPV4}, 
using Lemma \ref{calc1} (a) and (b), 
we see that 
$ \eqref{calc22} \lesssim \sup_{\xi, \tau} \frac{|\xi|^\frac{3}{4}}{\jb{\tau - \xi^3}^{1-b} \jb{4\tau - \al \xi^3}^\frac{1}{4}} \lesssim 1,$
\noindent
holds true for $\frac{1}{2} < b \leq \frac{3}{4}$ and $0 <\al <1$.  
Note that it is crucial to have $\al \ne 4$ (and 0), 
which guarantees that at least one of $\jb{\tau - \xi^3}$ or $\jb{4\tau - \al \xi^3}$ is $\sim |\xi|^3$
for any $\xi, \tau \in \R$.

\end{remark}

Now, we turn to the proof of \eqref{Rbilinear2}.
Define the bilinear operator $\wt{\mathcal{B}}_{s, b}$ by 
\[\wt{\mathcal{B}}_{s, b} \big(f, g\big) (\xi, \tau)= \frac{\xi \jb{\xi}^s }{\jb{\tau - \al \xi^3}^{1-b} }
\iintt_{\substack{\xi = \xi_1 + \xi_2 \\\tau = \tau_1+\tau_2}}
\frac{f(\xi_1, \tau_1) g(\xi_2, \tau_2)}{ \jb{\xi_1}^s\jb{\xi_2}^s\jb{\tau_1 - \xi_1^3}^b \jb{\tau_2-\al\xi_2^3}^b}d\xi_1d\tau_1.\]

\noindent
Then, \eqref{Rbilinear2} holds for $s = 0$ if and only if
$\big\| \wt{\mathcal{B}}_{0, b}(f, g) \big\|_{L^2_{\xi, \tau} }
\lesssim  \| f \|_{L^2_{\xi, \tau} } \| g \|_{L^2_{\xi, \tau}} . $
As before, by H\"older inequality, we have 
\begin{align*}
\big\| \wt{\mathcal{B}}_{0, b}(f, g) \big\|_{L^2_{\xi, \tau} } 
\leq \bigg\|\frac{|\xi|}{\jb{\tau - \al \xi^3}^{1-b} }
\bigg( \iint \frac{d \xi_1 d\tau_1}{\jb{\tau_1 -  \xi_1^3}^{2b}\jb{\tau_2 - \al \xi_2^3}^{2b}} \bigg)^\frac{1}{2} \bigg\|_{L^\infty_{\xi, \tau}} \| f \|_{L^2_{\xi, \tau} } \| g \|_{L^2_{\xi, \tau}} . 
\end{align*}

\noindent
Therefore, it suffices to prove
\begin{equation} \label{calc4}
\sup_{\xi, \tau} \frac{|\xi|}{\jb{\tau - \al \xi^3}^{1-b}} 
\bigg( \iint \frac{d \xi_1 d\tau_1}{\jb{\tau_1 -  \xi_1^3}^{2b}\jb{\tau_2 - \al \xi_2^3}^{2b}} \bigg)^\frac{1}{2} 
\lesssim 1.
\end{equation}

\noindent
Moreover,  by applying Lemma \ref{calc1} (a)  to the integration in $\tau_1$ of \eqref{calc4}, it also suffices to show
\begin{equation} \label{calc44}
\sup_{\xi, \tau} \frac{|\xi|}{\jb{\tau - \al \xi^3}^{1-b}} 
\bigg( \int \frac{d \xi_1 }{\jb{\tau- \xi_1^3 - \al \xi_2^3  }^{2b}} \bigg)^\frac{1}{2} 
\lesssim 1.
\end{equation}

\noindent
Hence, we  divide $\mathbb{R}^4 = \{ (\xi, \xi_1, \tau, \tau_1) \}$ into finitely many regions and prove that
either \eqref{calc4} or \eqref{calc44} holds in each of them.
If $|\xi| \lesssim 1$, then  for $\frac{1}{6} < b \leq 1$, we have
$  \frac{|\xi|}{\jb{\tau - \al \xi^3}^{1-b}} \lesssim 1$.
Also, by Lemma \ref{calc1} (c), 
$\int \frac{d \xi_1 }{\jb{\tau- \xi_1^3 - \al (\xi - \xi_1)^3  }^{2b}} \lesssim 1$.
Hence, \eqref{calc44} holds when $|\xi| \lesssim 1$.

For the following argument, assume $\xi \gtrsim 1$.  Then, for fixed $\xi > 0$ and $\tau$, let 
\[F(\xi_1) = \tau - \al (\xi - \xi_1)^3 - \xi_1^3 = \tau - \al\xi^3 - \big( (1-\al) \xi_1^3 + 3\al \xi \xi_1^2 - 3 \al \xi^2 \xi_1\big).\]

\noindent
Then, 
$F'(\xi_1) = - 3 \big( (1-\al) \xi_1^2 + 2\al \xi \xi_1 -  \al \xi^2 \big).$
By solving $F'(\xi_1) = 0$, we have $\xi_1 = r_1 \xi, r_2 \xi$, 
where $r_1 = \frac{\al^\frac{1}{2}}{1+ \al ^\frac{1}{2}} $ and $r_2 =  -\frac{\al^\frac{1}{2}}{1- \al ^\frac{1}{2}}$.
Then, it follows that $\mu = F(\xi_1)$ is monotone on each of 
$ (-\infty, r_2 \xi )$, $ [r_2 \xi, r_1 \xi)$, and $ [r_1 \xi, \infty). $
Now, suppose $|F'(\xi_1)| \gtrsim |\xi|^2$. Then, by  change of variables of $\mu = F(\xi_1)$ on each of the intervals above, we have, for $b > \frac{1}{2}$, 
\[ \bigg|\int \frac{d\xi_1}{\jb{\tau - \al (\xi - \xi_1)^3 - \xi_1^3}^{2b}} \bigg| \lesssim \frac{1}{|\xi|^2} \bigg| \int \frac{F'(\xi_1)}{\jb{F(\xi_1)}^{2b}} d\xi_1 \bigg|
= \frac{1}{|\xi|^2} \int \frac{1}{\jb{\mu}^{2b}} d\mu \lesssim \frac{1}{|\xi|^{2}}. \]

\noindent
Hence, \eqref{calc44} holds, if we assume $|F'(\xi_1)| \gtrsim |\xi|^2$.

Now, let $G(\xi_1)  = (1-\al) \xi_1^2 + 3\al \xi \xi_1 - 3 \al \xi^2$.
i.e.  $\xi_1 G(\xi_1) =  \al\xi^3 - \xi_1^3 - \al \xi_2^3$.
Suppose $|G(\xi_1)| \gtrsim |\xi|^2$
and  $|\xi_1| \sim |\xi|$.
Then, we have
\begin{equation} \label{calc5}
\MAX  :=  \max(\jb{\tau - \al \xi^3}, \jb{ \tau_1 - \xi_1^3 }, \jb{\tau_2 - \al \xi_2^3} )
 \gtrsim |\xi_1 G(\xi_1)| \gtrsim |\xi_1| |\xi|^2 \gtrsim |\xi|^3 \text{ on } A.
\end{equation}

\noindent
If $\MAX = \jb{\tau_1 - \xi_1^3}$, then  LHS of \eqref{calc4} is at most
\begin{align*}
 \lesssim 
\frac{|\xi|^{1-2b}} {\jb{\tau - \al \xi^3}^{1-b}} \bigg( \int \frac{d\xi_1}{\jb{\xi_1}^{2b}} \bigg)^\frac{1}{2}
\sup_{\xi_1} \bigg(\int \frac{d\tau_2}{\jb{\tau_2 - \al(\xi-\xi_1)^3}^{2b}}\bigg)^\frac{1}{2} 
 \lesssim \jb{\xi}^{1-2b} \leq 1,
\end{align*}

\noindent 
for $b > \frac{1}{2}$. Hence, \eqref{calc4} holds.
A similar computation shows that \eqref{calc4} holds if $\MAX = \jb{\tau_2 - \al \xi_2^3}$.
Lastly, if $\MAX = \jb{\tau - \al \xi^3}$, then by \eqref{calc5}, we have
$ \frac{|\xi|}{\jb{\tau - \al \xi^3}^{1-b}} \lesssim |\xi|^{-2+3b} \lesssim 1$
for $ b \leq \frac{2}{3}, $ 
and by Lemma \ref{calc1} (c), 
$ \int \frac{d\xi_1}{\jb{\tau - \xi_1^3 - \al (\xi - \xi_1)^3}^{2b}} \lesssim 1$
for $ b > \frac{1}{6}. $ 
Hence, \eqref{calc44} holds in this case.

Finally, we need to show that 
\begin{equation} \label{SUBdivision}
\{ (\xi, \xi_1, \tau, \tau_1): \xi \gtrsim 1\} \subset 
\{|F'(\xi_1)| \gtrsim |\xi|^2\} \cap \{|G(\xi_1)| \gtrsim |\xi|^2, |\xi_1|\sim |\xi| \}.
\end{equation}

\noindent
Consider
$H(\xi_1) := \tfrac{1}{3}F'(\xi_1) - G(\xi_1) = \al\xi\xi_1 - 2 \al \xi^2 = \al\xi(\xi_1 - 2\xi).$
Then,  we have $H(\xi_1) \leq - \al \xi^2$ for $\xi_1 \leq \xi$. 
Hence, we have 
$ \max \big( |F'(\xi_1)|, |G(\xi_1)| \big)\gtrsim |\xi|^2 $ for  $\xi_1 \leq \xi.$
Also, we have $|F'(\xi_1)| \gtrsim \xi^2$ for $\xi_1 \geq \xi > 0$.
Hence, we have $\max \big( |F'(\xi_1)|, |G(\xi_1)| \big)\gtrsim |\xi|^2 $ for any $\xi_1 \in \mathbb{R}$.
Now, suppose $|F'(\xi_1)| \ll |\xi|^2$.
Since $\al < 1$, $F'(\xi_1) $ is a downward parabola whose vertical intercept is $3 \al \xi^2$.
Then, one can easily show that $|\xi_1| \sim |\xi|$ in this case (say, by solving $|F'(\xi_1) |\leq \al |\xi|^2$.)
Hence, \eqref{SUBdivision} holds.
This completes the proof of \eqref{Rbilinear2}.
\end{proof}

Next, we  present the proof of Proposition \ref{PROP:Rcounterexample}.
The basic idea of the proof is the same as that for KdV in \cite{KPV4}.
Note that, unlike KdV, the resonances occur at non-symmetric points (e.g. at $(\xi_1, \xi_2) = (c_1 N, c_2N)$ for \eqref{Rbilinear1}),
and there are two curves $\tau = \xi^3$ and $\tau =\al \xi^3$ under consideration.
Thus, instead of taking the rectangles parallel to the curves, 
we take them to be parallel to the $\xi$- and $\tau$-axes.

First, recall the following estimate for the convolution of the characteristic functions of two {\it parallel} rectangles.
It follows from a straightforward computation and thus we  omit its proof.  See \cite{OHTHESIS}.

\begin{lemma} \label{rectangles}
Let $R$ and $ \wt{R}$ be rectangles centered at $(a, b)$ and $(\wt{a}, \wt{b})$
whose dimensions are $2\al \times 2\beta$.
Let $R_0$ be the parallel translate of $R$ centered at $(a+\wt{a}, b+\wt{b})$.
Then, we have
\begin{equation*}
\chi_R * \chi_{\wt{R}} (\xi, \tau) \geq \al \beta \chi_{R_0}(\xi, \tau) = \tfrac{1}{4} \Area (R) \chi_{R_0} (\xi, \tau).
\end{equation*}
\end{lemma}

\begin{proof}[Proof of Proposition \ref{PROP:Rcounterexample}]
 We only construct the counterexample to \eqref{Rbilinear1},
 since the construction of the counterexample to \eqref{Rbilinear2} is similar.
Recall that the resonance occurs in this case when $(\xi_1, \xi_2) = (c_1\xi, c_2 \xi)$. 
Thus, we would like to choose $f$ and $g$ such that
$ \supp f \sim \xi_1 = c_1N,$ $\supp g \sim \xi_2 = c_2N,$ and  $\supp f*g  \sim \xi = N. $

For large $N$, consider two rectangles $R_j$ of dimensions $\sim N^{-2} \times 1$ centered at $\big(c_j N, \al (c_jN)^3\big)$, $j = 1, 2$
such that the intersections of $\partial R_j$ and the curve $\tau = \al \xi^3$ are on the horizontal sides of $R_j$.
Note that the last condition can be satisfied for large $N$ since
the smallest slope of $\tau = \al \xi^3$ on $R_j$ is $ \sim 3\al c_j^2(N - N^{-2})^2 \sim N^2 $ for large $N$.
Now, let
$ f(\xi_1, \tau_1) = \chi_{R_1}$ and $ g(\xi_2, \tau_2) = \chi_{R_2}$. 
Then, we have $\| f\|_{L^2_{\xi, \tau}} = \| g\|_{L^2_{\xi, \tau}} \sim N^{-1}.$
On $R_j$, we have
$\jb{ \xi_j}^s \sim N^s$ and $\jb{\tau_j - \al \xi_j^3} \sim 1$.
Moreover, on $R_1 + R_2 = \supp f*g$, we have 
$\xi \sim N \ \text{ and } \ \jb{\tau - \xi^3} \sim 1 $,
since $N^3 = \al (c_1N)^3 + \al (c_2N)^3$.
Let $R_0$ be the rectangle centered at $(N, N^3)$ of the same size as $R_j$.
Then, we have $R_0 \subset R_1 + R_2$ and 
$ f*g(\xi, \tau)  \gtrsim N^{-2} \chi_{R_0}(\xi, \tau)$
by Lemma \ref{rectangles}.

Now, suppose that the bilinear estimate \eqref{Rbilinear1} holds.
i.e. we have
\begin{equation} \label{Rcounterexample1}
\big\| \mathcal{B}_{s, b}(f, g) \big\|_{L^2_{\xi, \tau}} \lesssim \| f\|_{L^2_{\xi, \tau}}\| g\|_{L^2_{\xi, \tau}},
\end{equation}

\noindent
where $\mathcal{B}_{s, b}(\cdot, \cdot)$ is defined in \eqref{Rdual}.
Then, for any large $N$, we have RHS of \eqref{Rcounterexample1} $\sim N^{-2}$, and
$ \text{LHS of }\eqref{Rcounterexample1} \gtrsim N^{1-s} N^{-2} \|\chi_{R_0} \|_{L^2_{\xi, \tau}} \sim N^{-2-s},$
which implies $s \geq 0$.
\end{proof}

\section{On the Ill-posedness Results}

There are also -so called- ``ill-posedness" results for dispersive equations such as the KdV equation \eqref{KDV}.
However, this term often refers to the necessary conditions for uniform continuity or smoothness of the solution map $\Phi_t : u_0 \in H^s \longmapsto u(t) \in H^s$. 
The Cauchy problem is {\it not necessarily ill-posed} in the sense of the usual definition, even when these results hold.
However, since the contraction argument provides smoothness of the solution map
it is often natural to consider a {\it strengthened}  notion of 
well-posedness requiring the solution map to be uniformly continuous/smooth.
In this latter sense, the following results may be regarded as ``ill-posedness" results. 

For KdV, Bourgain \cite{BO3} proved that if the solution map is $C^3$, then $s \geq -\frac{3}{4}$ on $\R$ and
$s \geq -\frac{1}{2}$ on $\T$.  
Tzvetkov \cite{TZ} improved Bourgain's result and showed that $C^2$ is enough on $\R$.
It was also shown that if the solution map is uniformly continuous, 
then $s \geq -\frac{3}{4}$ on $\R$ and
$s \geq -\frac{1}{2}$ on $\T$.
These results are obtained from the corresponding ill-posedness results of mKdV
and the (modified) Miura transform, which is not available for the Majda-Biello system \eqref{MB}.
See Kenig-Ponce-Vega \cite{KPV5} 
and Christ-Colliander-Tao \cite{CCT}.

When $\al = 1$, by setting $v_0 = \sqrt{2} u_0$ and $v(t) = \sqrt{2} u(t)$, we see that
\eqref{MB} reduces to a single KdV equation: $u_t + u_{xxx} + 2u u_x = 0$.
Hence, it follows from \cite{CCT} that the solution map  is not uniformly continuous
for $s < -\frac{3}{4} $ on $\R$ and for  $s < - \frac{1}{2}$ on $\T$.

Let $0 < \al < 1$. 
Following Bourgain \cite{BO3}, consider the following Cauchy problem:
\begin{equation} \label{illposedMB}
\begin{cases}
u_t + u_{xxx} + v v_x = 0 \\
v_t + \al v_{xxx} + (uv)_x = 0\\
\big( u(x, 0), v(x, 0) \big) = \big( \delta \phi(x) ,  \delta \psi(x) \big)
\end{cases}
\end{equation}

\noindent
where $\dl \geq 0$ and $x \in \mathbb{T}$ or $ \mathbb{R}$.
Let $\big( u(x, t;\dl), v(x, t;\dl) \big)$ and $\big( u(t;\dl), v(t;\dl) \big)$ denote the solution to \eqref{illposedMB}.
First, note that with $\dl = 0$, $\big( u(x, t; 0), v(x, t; 0) \big)  \equiv 0$ is the unique solution.
Also, by writing as integral equations, we have
\begin{equation*}
\begin{cases}
u(t;\dl) = \dl S(t) \phi - \int_0^t S(t - t') \dx \big( \frac{v^2}{2} \big) (t') dt' \\
v(t;\dl) = \dl S_\al(t) \psi - \int_0^t S_\al(t - t') \dx \big( uv \big) (t') dt' ,
\end{cases}
\end{equation*}

\noindent 
where $S(t) = e^{t\dx^3}$ and $S_\al(t) = e^{t\al \dx^3}$.
By taking derivatives in $\dl$ at $\dl = 0$,  we have 
$\dd u(t;0) = S(t) \phi =: \phi_1$ and 
$\dd v(t;0) = S_\al(t) \psi =: \psi_1$.
By taking the second and third derivatives in $\dl$ at $\dl = 0$, we have
\begin{align*}
& \begin{cases}
\dd^2 u(t; 0) = - \int_0^t S(t - t') \dx \big( \psi_1 \big) (t') dt'  =: \phi_2 \\
\dd^2 v(t; 0) = - \int_0^t S_\al(t - t') \dx \big( 2 \phi_1 \psi_1 \big) (t') dt'  =: \psi_2, 
\end{cases}
\\ & \begin{cases}
\dd^3 u(t; 0) =  \int_0^t S(t - t') \dx \big( 3 \psi_1 \psi_2 \big) (t') dt'  =: \phi_3 \\
\dd^3 v(t; 0) =  \int_0^t S(t - t') \dx \big( 3 \phi_1 \psi_2+ 3 \phi_2 \psi_1 \big) (t') dt'  =: \psi_3 .
\end{cases}
\end{align*}

Note that if the solution map $\Phi_t: (u_0, v_0)  \in H^s \times H^s 
\longmapsto \big( u(t), v(t) \big) \in H^s \times H^s$ is $C^k$ for fixed $|t| \ll 1$, then we must have
\begin{equation} \label{C^ksmooth}
\big\| \dd^k \big(u, v\big) (\cdot, t; 0)   \big\|_{H^s_x \times H^s_x} 
= \| (\phi_k, \psi_k) (\cdot, t) \|_{H^s_x \times H^s_x}
\lesssim \left\| (\phi, \psi ) \right\|^k_{H^s\times H^s} 
\end{equation}

\noindent
from the smoothness of $\Phi_t$ at the zero solution. 
This yields the ill-posedness results on $\mathbb{T}$ and $\R$:
Theorems \ref{THM:illposedonT}  and  \ref{THM:illposedonR}.
In the following, we assume that \eqref{MB} is  well-posed in $H^s \times H^s$ 
over a small time interval and fix $|t| \ll 1$ such that the solution map $\Phi_t$ is well-defined.

\begin{proof}[Proof of Theorem \ref{THM:illposedonT}]

We only prove the results involving $c_1$.
The results for $d_1$ and $d_2$ can be proved analogously.

\noindent
$\bullet$ {\bf Case (1):} $c_1 \in \mathbb{Q}$

In this case, we can choose $N\in \mathbb{N} \to \infty$ such that $c_1 N \in \mathbb{N}$.
For fixed such $N$, let $\phi \equiv 0$ and  
$\psi(x) = N^{-s} \big(\cos(c_1Nx) + \cos(c_2Nx) \big).$
Then, we have $\| (\phi, \psi)\|_{H^s \times H^s} = \|\psi\|_{H^s} \sim 1$.
A direct computation shows
$ \psi_1 (x, t)  = N^{-s} \big(\cos(c_1Nx + \al (c_1N)^3t) + \cos(c_2Nx + \al (c_2N)^3t) \big)$
and thus
\begin{align*}
S(t - t') \dx \big(\psi_1^2\big) (x, t') 
& = -N^{-2s + 1} 
\Big[  \sin \big(N x + N^3 t + (\al c_1^3  + \al c_2^3  - 1)N^3 t' \big) \\
& +c_1 \sin \big( 2 c_1 N x   + 8 (c_1N)^3 t + 2 (\al - 4 ) (c_1N)^3 t' \big)  \\
&+ c_2 \sin \big( 2 c_2 N x  + 8 (c_2N)^3 t + 2 (\al - 4 ) (c_2N)^3 t' \big)  \\
&+ c_3 \sin \big( c_3 Nx + (c_3 N)^3t + (\al c_1^3  - \al c_2^3 - c_3^3 )N^3t' \big)  \Big],
\end{align*}

\noindent
where $c_3 = c_1 - c_2$.
From \eqref{resonance11}, we see that the coefficient of $t'$ in the first term is zero for any $N$. 
Also, note that $\al - 4 < 0$ and 
$(\al c_1^3  - \al c_2^3 - c_3^3 ) N^3
 = O(N^3)$
 for $ 0 < \al < 1$.
Thus, we have
$\phi_2 (x, t)
=  t N^{-2s+1}  \sin ( Nx + N^3t) + O( N^{-2s -2}  )$.
Hence if we assume the solution map $\Phi_t$ is $C^2$, then from \eqref{C^ksmooth},
we have
\[ N^{-s+1} \sim  \| \phi_2 (\cdot, t)\|_{H^s_x}  \leq \| (\phi_2, \psi_2)(\cdot, t) \|_{H^s_x \times H^s_x} \lesssim \| (\phi, \psi) \|^2_{H^s \times H^s} \sim 1\]

\noindent
for any $N$ satisfying $c_1 N \in \mathbb{N}$.
Hence, we must have $s \geq 1$, if the solution map $\Phi_t$ is $C^2$.

\noindent
$\bullet$ {\bf Case (2):} $c_1 \in \mathbb{R} \setminus \mathbb{Q}$

Given $N \in \mathbb{N}$,  let $\phi \equiv 0$ and  
$\psi(x) = N^{-s} (\cos([c_1N]x) + \cos ( [c_2N]x ) )$,
where $[x] =$ the closest integer to $x$ as in the proof of Proposition  \ref{PROP:counterexample1}.
From the computation in Case (1) replacing $c_jN$ with $[c_j N]$, we have
\begin{align} \label{phi2}
\phi_2 (x, t) 
& =  N^{-2s+1} \frac{ \cos ( Nx + \al [c_1N ]^3 t + \al [c_2N ]^3 t) 
- \cos ( Nx  +N^3t) }{\al [c_1N ]^3  + \al [c_2N ]^3 - N^3}
+ O( N^{-2s -2}  ).
\end{align}

If $\nu_{c_1} > 1$, then
it follows from \eqref{resres11} that for there are infinitely many $N$ such that
$\big| \al [c_1N ]^3  + \al [c_2N ]^3 - N^3 \big| \ll 1.$
Then, by Mean Value Theorem, 
we have
$\phi_2 (x, t) \sim - t N^{-2s+1} \sin (Nx + N^3t ) +  O( N^{-2s -2}  ),$
and hence, we have $\| \phi_2(\cdot, t)\|_{H^s_x} \sim N^{-s+1}$ for infinitely many $N$.
This implies $s \geq 1$, if $\nu_{c_1} > 1$ and the solution map $\Phi_t$ is $C^2$.

On the other hand, it follows from \eqref{resres111} and \eqref{phi2} that for any $\eps >0$, we have
$\| \phi_2(\cdot, t)\|_{H^s_x} \lesssim N^{-s +\nu_{c_1} + \eps}$
for all sufficiently large $N \in \mathbb{N}$.
Thus, if $\nu_{c_1} \leq 1$, then we can not deduce any condition on $s$ (especially when $s \geq \nu_{c_1}$)
even if we assume that the solution map $\Phi_t$ is $C^2$.

Now, assume that $\nu_{c_1} \leq 1$ and that the solution map $\Phi_t$ is $C^3$.
Since $\phi \equiv 0$, we have $\phi_1 \equiv 0$ and $\psi_2 \equiv 0$.
Thus, we have 
$\psi_3 (t) = 3 \int_0^t S_\al (t - t') \dx \big(\phi_2 \psi_1\big) (t') dt'.$
From \eqref{phi2}, we have
\begin{align} \label{phi2psi1}
 \phi_2 \psi_1(x, t) & = 
\frac{N^{-3s+1}}{ 2(\al [c_1N ]^3 t + \al [c_2N ]^3 - N^3)} 
 \times \Big[
\sum_{j = 1}^2 \cos ( [c_j N]x + \al [c_jN ]^3 t) \notag \\
& \hphantom{XX} + \text{ 6 more terms with cosines } \Big] 
+ \text{ lower order terms}.
\end{align}

\noindent
Then, from \eqref{resres11}, it follows  that for any $\eps > 0$, there are infinitely many $N$ such that
\begin{align} \label{psi3}
\psi_3 (x, t) 
&\sim \frac{t N^{-3s+2}} {\al [c_1N ]^3  + \al [c_2N ]^3 -N^3}  
\sum_{j = 1}^2 \sin ( [c_j N]x + \al [c_jN ]^3 t)  
+ \text{lower order terms }\notag \\
& \gtrsim t N^{-3s+1 + \nu_{c_1} - \eps} 
\sum_{j = 1}^2 \sin ( [c_j N]x + \al [c_jN ]^3 t)  . 
\end{align}

\noindent
Hence, we have
$\| \psi_3(\cdot, t) \|_{H^s_x} \sim N^{-2s+1 + \nu_{c_1} - \eps}. $
By letting $N \to \infty$, this implies that $-2s+1 + \nu_{c_1} - \eps \leq 0$ for any $\eps > 0$, if the solution map $\Phi_t$ is $C^3$.
Hence,  if the solution map is $C^3$, then we must have $s \geq \frac{1}{2} + \frac{1}{2}\nu_{c_1}$.
We point out that a careful examination of \eqref{phi2psi1} and \eqref{psi3} shows 
$\| \psi_3(\cdot, t) \|_{H^s_x} \lesssim N^{-2s+2}. $
i.e. this part of the argument is only for 
 $\nu_{c_1} \leq 1$.
\end{proof}

\begin{proof}[Proof of Theorem \ref{THM:illposedonR}]

Let $\phi \equiv 0$ and
\[\psi (x) = 2\g^{-\frac{1}{2}}N^{-s} \Big( \cos(c_1 N x) \int_{-\g}^\g e^{i \xi x} d\xi +
+ \cos(c_2 N x) \int_{-\g}^\g e^{i \xi x} d\xi \Big),\]
where $\g = \eps N^{-{2}}$ and we  choose $\eps$ later. Then, 
 $\| (\phi, \psi) \|_{H^s \times H^s} = \| \psi \|_{H^s} \sim 1$. 
Then, \[\psi_1 = S_\al(t) \psi = \g^{-\frac{1}{2}}N^{-s} \Big( \intt_{|\xi \pm c_1N| <\g} e^{i (\xi x+ \al \xi^3t)} d\xi +
\intt_{|\xi \pm c_2N| <\g} e^{i (\xi x+ \al \xi^3t)} d\xi 
\Big).\]

\noindent
Let $ A = \{ (\xi_1, \xi_2) \in \mathbb{R}^2 : |\xi_1 \pm c_1N| <\g,  \ |\xi_2 \pm c_2N| <\g\}$ and $\xi = \xi_1 + \xi_2$.
Then, we have
\[ \phi_2(x, t) 
 = \g^{-1}N^{-2s}  \iint_A i \xi e^{i\xi x} e^{i\xi^3t}\frac{e^{i(\al \xi_1^3 + \al \xi_2^3 - \xi^3) t} - 1}
{\al \xi_1^3 + \al \xi_2^3 - \xi^3}d\xi_1 d\xi_2.\]

\noindent
Note that $\| \phi_2(\cdot, t)\|_{H^s_x} = \big\| \jb{\xi}^s |\ft{\phi_2}(\xi, t)| \big\|_{L^2_\xi} 
\geq \big\| \jb{\xi}^s |\ft{\phi_2}(\xi, t)| \chi_{|\xi | = N + O(\g)} \big\|_{L^2_\xi} $.
Thus, by restricting our attention to $\{ \xi \in \mathbb{R} : \xi = N + O(\g) \}$ in 
\begin{align*}
\ft{\phi_2}(\xi, t) 
= \g^{-1}N^{-2s} i \xi  \, e^{i\xi^3t} 
\intt_{\substack{ \{\xi = \xi_1 + \xi_2\}\cap A}} \frac{e^{i(\al \xi_1^3 + \al \xi_2^3 - \xi^3) t} - 1}
{\al \xi_1^3 + \al \xi_2^3 - \xi^3}d\xi_1,
\end{align*}

\noindent
we see that  the only contribution to $\{ \xi \in \mathbb{R} : \xi = N + O(\g) \}$
comes from $B = \{ \xi = \xi_1 + \xi_2, \ |\xi_1 - c_1N| <\g$ and $|\xi_2 - c_2N| <\g\} $ since $c_1 + c_2 = 1$.
Note that for sufficiently large $N$, we have
\[\begin{cases}|\xi_j^3 - (c_j N)^3| = |(\xi_j - c_jN)(\xi_j^2 + \xi_j c_jN + (c_j N)^2| \lesssim \g N^2  \sim \eps \\
|\xi^3 - N^3| = |(\xi - N) (\xi^2 + \xi N + N^2)| \lesssim \g N^2  \sim \eps 
\end{cases}\]

\noindent
on $B$. Then, using $\al (c_1N)^3 + \al (c_2N)^3 - N^3 = 0$, we have
$ |\al \xi_1^3 +\al \xi_2^3 - \xi^3| 
= |\al (\xi_1^3 - (c_1N)^3) + \al (\xi_2^3 - (c_2N)^3) - (\xi^3 - N^3)| \lesssim \eps.$
Since $\lim_{\theta \to 0} \frac{e^{i\theta t} - 1}{\theta } = it$, 
we have
 $\text{Im} \frac{e^{i\theta t} - 1}{\theta } \geq \frac{t}{2}$ for sufficiently small  $\theta $.
Now, by choosing $\eps$ small such that the above inequality holds with $|\theta| = |\al \xi_1^3 +\al \xi_2^3 - \xi^3| \lesssim \eps$, we have
\begin{align*}
\big|\ft{\phi_2}(\xi, t) \chi_{|\xi | = N + O(\g)} \big| 
\sim \g^{-1}N^{-2s+1}    \bigg|\intt_{B} \frac{e^{i(\al \xi_1^3 + \al \xi_2^3 - \xi^3) t} - 1}
{\al \xi_1^3 + \al \xi_2^3 - \xi^3}d\xi_1\bigg|
\geq  \g^{-1}N^{-2s+1}   \frac{t}{2}\int_B 1 \, d \xi_1 .
\end{align*}

\noindent
From Lemma \ref{rectangles}, we have
$\int_B 1 \, d \xi_1 = \chi_{|\ \cdot\ - c_1 N| < \g}*\chi_{|\ \cdot\ - c_2 N| < \g}(\xi ) 
\geq \tfrac{1}{2}  \g \chi_{|\xi - N | < \g}(\xi). $
Hence, we have
$ \| \phi_2( \cdot, t) \|_{H^s_x} 
\gtrsim t N^{-s + 1} \| \chi_{|\xi - N | < \g}(\xi) \|_{L^2_\xi} \sim t \g^\frac{1}{2}N^{-s + 1}
\sim N^{-s}. $
Therefore, 
it follows from \eqref{C^ksmooth} that we have $s\geq 0$ if the solution map $\Phi_t$ is $C^2$.
\end{proof}

\section{APPENDIX: Local Well-Posedness Result on $\T_\ld$ for $\al= 1$ without the Mean 0 Assumption}

Assuming the mean 0 condition for $u$ and $v$, 
the bilinear estimate \eqref{KPVZbilinear} for KdV on $\mathbb{T}_\ld$ along with the standard computation 
establishes the local well-posedness of the Majda-Biello system \eqref{MB} in $H^s(\mathbb{T}_\ld) \times H^s(\mathbb{T}_\ld)$ for $s \geq -\frac{1}{2}.$
In this appendix, we  establish the same result without the mean assumption on $u$ and $v$.

If the spatial means of $u$ and $v$ are not zero, we consider $u \mapsto u - (2\pi\ld)^{-1}\ft{u_0}(0)$
and $v \mapsto v - (2\pi\ld)^{-1} \ft{v_0}(0)$ along with the conservation $E_1$ and $E_2$ of the means of $u$ and $v$.
This modifies the  Majda-Biello system into the mean 0 system:
\begin{equation} \label{mean0MB}
\begin{cases}
u_t + u_{xxx} + q v_x + v v_x = 0 \\
v_t + v_{xxx} + q u_x + pv_x + (u v)_x = 0, 
\end{cases}
\end{equation}

\noindent
where $ p $ and $q$ are the spatial means of the original $u$ and $v$.
Now, consider the linear part of \eqref{mean0MB}:
\begin{equation} \label{mean0linear}
\bigg( \dt  + \dx^3  + \begin{pmatrix} 0 & q \\ q& p \end{pmatrix} \bigg) 
\begin{pmatrix} u \\ v \end{pmatrix} = 0.
\end{equation}

\noindent
When $q \ne 0$, the linear terms are {\it mixed}.
In this case, it does not make sense to consider the solution space as a product of the scalar $X^{s, b} $ spaces.
By taking the space-time Fourier transform of  \eqref{mean0linear}, we see that
the Fourier transforms of free solutions are ``supported on'' $\tau I - A(\xi)$, 
where $I $ is the $2\times 2$ identity matrix and 
$A(\xi ) = \Big(\begin{smallmatrix} \xi^3 & -q \xi \\ -q \xi & \xi^3 - p \xi \end{smallmatrix} \Big).$
Since $A(\xi) $ is self-adjoint, it is diagonalizable via an orthogonal matrix $M(\xi)$
(with $M(0) := I$.)
i.e.
we have $A(\xi) = M(\xi)D(\xi)M^{-1}(\xi) $, where 
$D(\xi) =  \Big(\begin{smallmatrix} d_1(\xi) & 0 \\ 0 & d_2(\xi) \end{smallmatrix} \Big)$  
and $ d_1(\xi) $, $d_2(\xi)$ are the eigenvalues of $A(\xi)$ given by
\begin{equation} \label{XEIGENVALUE}
d_j(\xi) = \xi^3 -\tfrac{p \xi}{2} + (-1)^j L\xi, \ j = 1, 2, 
\end{equation}

\noindent
with $L := L(p, q) = \frac{1}{2}\sqrt{ p^2 + 4q^2}$.
Then, we define the {\it vector-valued} $X^{s, b}$ space as follows:

\begin{definition}
Define $X^{s, b}_{p, q} (\mathbb{T}_\ld \times \mathbb{R}) = \big\{ (u, v) \in \mathcal{S}'  : \big\|(u, v) \big\|_{X^{s, b}_{p, q}} < \infty \big\}$,
via the norm
\begin{align*}
\|(u, v) \|_{X^{s, b}_{p, q}(\mathbb{T}_\ld \times \mathbb{R})} & = \left\| \jb{\xi}^s \big( I + |\tau - A(\xi)|\big)^b 
\left( \begin{smallmatrix}  \ft{u}(\xi, \tau) \\ \ft{v}(\xi, \tau) \end{smallmatrix} \right) \right\|_{L^2_{\xi, \tau}(\mathbb{Z}/\ld \times \mathbb{R})} \\
& = \bigg( \frac{1}{2\pi \ld} \int \sum_{\xi \in \mathbb{Z}/\ld} 
\Big[ \jb{\xi}^{2s} \big( I + |\tau - A(\xi)|\big)^{2b} 
\left( \begin{smallmatrix}  \ft{u}(\xi, \tau) \\ \ft{v}(\xi, \tau) \end{smallmatrix} \right),
\left( \begin{smallmatrix}  \ft{u}(\xi, \tau) \\ \ft{v}(\xi, \tau) \end{smallmatrix} \right) 
\Big]_{\mathbb{C}^2} d \tau \bigg)^{1/2},
\end{align*}

\noindent
where $ [\cdot, \cdot]_{\mathbb{C}^2}$ is the usual Euclidean inner product on $\mathbb{C}^2$.
\end{definition}

\noindent We  drop the subscripts $p$ and $q$ when there is no confusion.

\begin{remark} \label{XREMARK1} \rm
Since $\tau I - A(\xi)$ is self-adjoint, $\big(\tau I - A(\xi)\big)^2$ is a positive matrix, 
with a unique positive  square root. We define $|\tau I - A(\xi)|$ by such a unique square root. 
Then, $ I + |\tau I - A(\xi)| $ is also positive definite and we can define $\big(I + |\tau I - A(\xi)| \big)^{2b}$ by
$M(\xi) \big(I + |\tau I - D(\xi)| \big)^{2b} M^{-1}(\xi) $.
\end{remark}

\begin{remark} \label{XREMARK2} \rm
Note that the $X^{s, b}_{p, q}$ norm is {\it not} defined as a weighted $L^2$ norm of $|\ft{u}|$ and $|\ft{v}|$, 
unlike the scalar $X^{s, b}$ norm (which is a weighted $L^2$ norm of $|\ft{u}|$.)
Since $M(\xi)$ is an orthogonal matrix for all $\xi$, we have
$\left| M^{-1}(\xi) \left( \begin{smallmatrix}  \ft{u}(\xi, \tau) \\ \ft{v}(\xi, \tau) \end{smallmatrix} \right) \right|_{\mathbb{C}^2}
= \left| \left( \begin{smallmatrix}  \ft{u}(\xi, \tau) \\ \ft{v}(\xi, \tau) \end{smallmatrix} \right) \right|_{\mathbb{C}^2}$
for all $\xi \in \mathbb{Z}/\ld$ and $\tau \in \mathbb{R}$.
Thus we can take the inverse Fourier transform of 
$M^{-1}(\xi) \left( \begin{smallmatrix}  \ft{u}(\xi, \tau) \\ \ft{v}(\xi, \tau) \end{smallmatrix} \right) $. 
Then, by letting
\begin{equation} \label{XDIAGONAL}
 \left( \begin{smallmatrix}  \ft{U}(\xi, \tau) \\ \ft{V}(\xi, \tau) \end{smallmatrix} \right) =
M^{-1}(\xi) \left( \begin{smallmatrix}  \ft{u}(\xi, \tau) \\ \ft{v}(\xi, \tau) \end{smallmatrix} \right) ,
\end{equation}

\noindent 
we have
\begin{align*}
\|(u, v) & \|_{X^{s, b}_{p, q}(\mathbb{T}_\ld \times \mathbb{R})}
 =\bigg( \iint \jb{\xi}^{2s} \Big|\jb{\tau I - A(\xi) }^b \left( \begin{smallmatrix}  \ft{u}(\xi, \tau) \\ \ft{v}(\xi, \tau) \end{smallmatrix} \right) 
\Big|^2_{\mathbb{C}^2} d\xi^\ld d\tau\bigg)^\frac{1}{2} \\
& =\bigg( \iint \jb{\xi}^{2s} \Big| \jb{\tau I - D(\xi) }^b \left( \begin{smallmatrix}  \ft{U}(\xi, \tau) \\ \ft{V}(\xi, \tau) \end{smallmatrix} \right) 
\Big|^2_{\mathbb{C}^2} d\xi^\ld d\tau\bigg)^\frac{1}{2} 
 = \big( \| U \|^2_{X^{s, b}_1} + \| V\|^2_{X^{s, b}_2} \big)^{1/2}, 
\end{align*}

\noindent
where 
\begin{equation}\label{XXSB}
\| f \|_{X^{s, b}_j} = \big\| \jb{\xi}^s \jb{\tau - d_j (\xi) }^b \ft{f}(\xi, \tau)\big \|_{L^2_{\xi, \tau}}, \ j = 1, 2.
\end{equation}

\noindent
i.e. the $X^{s, b}_{p, q}$ norm 
 is defined as a weighted $L^2$ norm of the {\it diagonal} terms $|\ft{U}|$ and $|\ft{V}|$.
Hence, we can assume that $\ft{U}$ and $\ft{V}$ are nonnegative
in proving the estimates. 
\end{remark}

Using $X^{s, b}_{p, q}$ and other related function spaces, 
we establish the following LWP result.
In this appendix, we  only sketch the proof pointing out the difference from the scalar KdV case.
For details, see \cite{OHTHESIS}.

\begin{theorem} \label{THM:XLWP}
Let $\ld \geq 1$ and $p, q \in \mathbb{R}$.
The Cauchy problem \eqref{mean0MB} with mean 0 initial data $(u_0, v_0)$ is 
locally well-posed in $H^s(\mathbb{T}_\ld) \times H^s(\mathbb{T}_\ld)$ for $s \geq -\frac{1}{2}$. 
\end{theorem}

\noindent
Then, we obtain 
Theorem \ref{MTHM:mean0LWP} as a corollary.
As in the scalar case, 
the proof of Theorem \ref{THM:XLWP} is based on a contraction argument in $Y^s_{p,q} \subset C_t H^s_x $,
where 
\[\|(u, v) \|_{Y^s_{p,q}} = \| (u, v) \|_{X^{s,\frac{1}{2}}_{p, q}} + \| \jb{\xi}^{s}(\ft{u}, \ft{v}) (\xi, \tau)\|_{L^2(d\xi^\ld, L^1_{\tau})}.\] 

\noindent
Let $S(t)$ be the linear semigroup for the linear system \eqref{mean0linear}.  i.e. $S(t)$ is defined via
$\big( S(t) (u_0, v_0)^T   \big)^{\wedge} (\xi) 
= e^{itA(\xi)} (\ft{u_0}(\xi), \ft{v_0}(\xi))^T $,
where $T$ denotes the transpose.
As in \cite{BO1}, the Duhamel term 
$ \int_0^tS(t - t') \big(f(t'), g(t')\big)^T dt' $ can be written as
\begin{align*}
\int_0^t S(t -t') \col{f(x, t')}{g(x, t')} dt 
& = -i \iint e^{ix \xi} M(\xi) 
\col{(\tau - d_1(\xi))^{-1}
(e^{it \tau } - e^{itd_1(\xi)})\ft{F}(\xi, \tau)}
{ (\tau - d_2(\xi))^{-1}(e^{it \tau } - e^{itd_2(\xi)})\ft{G}(\xi, \tau)} d\tau d\xi^\ld  \\
& = -i \iint e^{ix \xi} ( \tau I - A(\xi) )^{-1} (e^{i t \tau } I - e^{itA(\xi)})
\col{\ft{f}(\xi, \tau)}{ \ft{g}(\xi, \tau)} d\tau d\xi^\ld,  
\end{align*}

\noindent
where $(\ft{F}, \ft{G})^T = M^{-1} (\ft{f}, \ft{g})^T$.
This computation leads us to define $ \| (u, v) \|_{Z^s_{p, q}} $ via
\[ \| (u, v) \|_{Z_{p, q}^s} = \| (u, v) \|_{X_{p, q}^{s,-\frac{1}{2}}} + \left\| \jb{\xi}^s \jb{ \tau I -A(\xi) } ^{-1}  
(u, v)^T \right\|_{L^2(d\xi^\ld, L^1_{\tau} ) }. \]

Here are some basic properties of $X^{s,b}_{p,q}$,  $Y^s_{p, q}$, and $Z_{p, q}^s$.
Their proofs are straightforward modifications from the scalar case. 
However, we always need to reduce the estimates for $(u, v)$ to those for the diagonal terms $(U, V)$ given by \eqref{XDIAGONAL}.
For details, see \cite{OHTHESIS}.

\begin{lemma}[Linear Estimates] \label{XFREESOLN}
\[ \begin{cases}
 \| \eta(t) S(t) (u_0,v_0)^T \|_{Y^s_{p, q}} \lesssim \|(u_0, v_0)\|_{H^s \times H^s} \\
 \big\|  \eta(t) \int_0^t S(t -t') (f(t'),g(t'))^T dt \big\|_{Y^{s}_{p, q}}  \lesssim  \| (f, g) \|_{Z_{p, q}^s}. 
\end{cases}\]
\end{lemma}

\noindent
The following lemmata are for the $X_j^{s, b}$ spaces defined in   \eqref{XXSB}.

\begin{lemma} \label{XEMBED1}
Let $f(x, t)$ be a function on $ \mathbb{T}_\ld \times \mathbb{R} $.
Then, we have 
\begin{equation*} 
 \| f \|_{L^4_t L^2_x} \lesssim \| f \|_{ X_j^{0, \frac{1}{4}}},  \ j = 1, 2.
\end{equation*}
\end{lemma}

\begin{lemma} \label{XEMBED2}
Let $\ld \geq 1 $ and $\g = \max(C/\ld, 1) $.
Let $f(x, t)$ be a function on $ \mathbb{T}_\ld \times \mathbb{R} $
such that $\supp \ft{f}(\xi, t) \subset [1/\ld, \g] \text{ for all } t \in \mathbb{R}$.
Then, we have
\begin{equation*}
\| f \|_{L^4_t L^\infty_x} \lesssim \ld^{0+} \big\| | \dx |^\frac{1}{2} f \big\|_{X_j^{0, \frac{1}{4}}}, 
\text{ and }
\| f \|_{L^2_t L^\infty_x} \lesssim \ld^{0+} \big\| | \dx |^\frac{1}{2} f \big\|_{ L^2_{x, t}}, \ j = 1, 2. 
\end{equation*}
\end{lemma}

Before discussing the important estimates, 
we'd like to discuss the scaling on the mean 0 system \eqref{mean0MB}
on $[0, 2 \pi \ld) \times \mathbb{R}$.
\eqref{mean0MB} was obtained by $u \to u - p$ and $v \to v - q$ from the Majda-Biello system \eqref{MB},
where $p$ and $q$ are the spatial means of the original $u$ and $v$ on $[0, 2\pi \ld)$, respectively.
Now, consider the scaling $\mathbb{T}_\ld = [0, 2\pi\ld) \mapsto \mathbb{T}_{\s\ld} = [0, 2\pi \s\ld)$ 
on  \eqref{MB} given by
\begin{equation*} 
u^\s(x, t) = \tfrac{1}{\s^2}u(\tfrac{t}{\s^3}, \tfrac{x}{\s}), 
\text{ and }
v^\s(x, t) = \tfrac{1}{\s^2}v(\tfrac{t}{\s^3}, \tfrac{x}{\s})
\end{equation*}

\noindent
Note that the scaling does {\it not} preserve the means of $u$ and $v$.
Rather, we have 
$p^\s =$ the mean of $u^\s  =  p / \s^2$,  and $q^\s =$ the mean of $v^\s = q / \s^2. $
Then, after scaling, we need to consider the following equation rather than \eqref{mean0MB}.
\begin{equation*}
\begin{cases}
 u^\s_t +  u^\s_{xxx} + q^\s v^\s_x + v^\s v^\s_x = 0\\
 v^\s_t +   v^\s_{xxx} + q^\s u^\s_x + p^\s  v^\s_x +  (u^\s v^\s)_x = 0 
\end{cases}
\end{equation*}

\noindent
on $[0, 2 \pi \s \ld) \times \mathbb{R}$,
where 
$p^\s = p / \s^2$ and $q^\s = q / \s^2.$
Hence, $p$ and $q$ in the definition of $X^{s, b}_{p, q}$, $Y^s_{p, q}$, and $Z^s_{p, q}$ need to be modified accordingly when we apply scaling.
i.e. we need to consider $X^{s, b}_{p^\s, q^\s}([0, 2\pi \s \ld) \times \mathbb{R})$ and so on.

Now, we state the $L^4$ Strichartz estimate for the vector-valued $X^{s, b}_{p, q}$ spaces, following Bourgain \cite{BO1}.
The proof is a straightforward modification of the argument in \cite{BO1}, and hence is omitted.
See \cite{OHTHESIS} for the full proof.

\begin{lemma} \label{LEM:XL4}
For a vector-valued function $(u, v) $ on $[0, 2\pi \ld ) \times \mathbb{R} $, 
\[ \| (u, v) \|_{L^4_{x, t}([0, 2\pi \ld ) \times \mathbb{R})} 
\lesssim \| (u, v) \|_{X^{0,\frac{1}{3}}_{p,q}([0, 2\pi \ld ) \times \mathbb{R})}. \] 
\noindent Moreover, the implicit constant $ C = C(\ld)$ is a non-increasing function of $\ld$. 

\end{lemma}

\noindent
The proof of Lemma \ref{LEM:XL4} immediately provides the following corollary for the diagonal terms $U_1$ and $U_2$ 
defined by 
$ (\ft{U}_1, \ft{U}_2)^T = M^{-1} (\ft{u}, \ft{v})^T.$

\begin{corollary} \label{COR:XL4}
 Let $X_j^{s, b}$ be as in \eqref{XXSB},  $j = 1, 2$.  Then, we have
\[
\| U_j \|_{L^4_{x, t}(\mathbb{T}_\ld \times \mathbb{R})} 
\lesssim \| U_j \|_{X^{0, \frac{1}{3}}_j (\mathbb{T}_\ld \times \mathbb{R})}
\leq \|(u, v) \|_{X^{0, \frac{1}{3}}_{p, q}(\mathbb{T}_\ld \times \mathbb{R})}, \  j = 1, 2. \]

\end{corollary}

\noindent
Lastly, we state the crucial bilinear estimate for establishing the local well-posedness result.
\begin{proposition} \label{PROP:XBILINEAR}
Let $\ld \geq 1$. 
Let $B(\cdot, \cdot)$ be a bilinear operator defined by
\begin{equation} \label{XBILINEAROP}
B( \vec{u}, \vec{v}) = B \left(
\left( \begin{smallmatrix} u_1 \\ u_2 \end{smallmatrix} \right),
\left( \begin{smallmatrix}  v_1 \\  v_2 \end{smallmatrix} \right)
\right)
= \left(  \tfrac{1}{2}u_2 v_2, \tfrac{1}{2} (u_1 v_2 + u_2 v_1) \right).
\end{equation}

\noindent
Then,  for mean 0 functions $\vec{u}$ and $\vec{v}$ on $\mathbb{T_\ld} \times \mathbb{R}$, we have,
for $s \geq - \frac{1}{2}$,
\begin{equation*} \label{XBILINEAR}
 \big\|  \dx \big( B( \vec{u}, \vec{v}) \big) \big\|_{Z_{p, q}^{s}} \lesssim \ld^{0+} 
\| \vec{u} \|_{X^{s, \frac{1}{2}}_{p, q} }\| \vec{v} \|_{X_{p, q}^{s, \frac{1}{2}}}. 
\end{equation*}

\end{proposition}

Like other lemmata, in proving Lemma \ref{LEM:XL4} and Proposition \ref{PROP:XBILINEAR}, 
we need to reduce the proofs for $(u, v)$ to those for the diagonal terms $(U, V)$ given by \eqref{XDIAGONAL} 
so that we can assume that $\ft{U}$  are $\ft{V}$ nonnegative.
Then, the main modification comes from estimating the sum and difference of the eigenvalues $d_1(\xi)$ and $d_2(\xi)$ of $A(\xi)$
and their scaling property.
To illustrate this, we  briefly discuss the proof of Proposition \ref{PROP:XBILINEAR} following 
Colliander-Keel-Staffilani-Takaoka-Tao \cite{CKSTT4}.

For simplicity, let $s = -\frac{1}{2}$.
First, consider the $X^{-\frac{1}{2}, -\frac{1}{2}} $ part of the $Z^{-\frac{1}{2}} $ norm.
By duality and an integration by parts, 
it suffices to prove
\begin{equation} \label{XDUALBILINEAR}
 \bigg|  \iint \Big[  \big(B( \vec{u}, \vec{v})\big) (x, t), \ \vec{w} (x, t) \Big]_{\mathbb{R}^2} dx dt \bigg| 
\lesssim \ld^{0+}  \| \vec{u} \|_{X^{-\frac{1}{2},\frac{1}{2}}}  \| \vec{v} \|_{X^{-\frac{1}{2},\frac{1}{2}}}
\| \vec{w} \|_{X^{-\frac{1}{2},\frac{1}{2}}} 
\end{equation}

\noindent
for all $\vec{u}, \vec{v}, \vec{w} $ with the spatial mean 0.
Let
$\ft{\vec{U}} = M^{-1} \ft{\vec{u}}$, 
$\ft{\vec{V}} =$ $ M^{-1} \ft{\vec{v}} $, 
and $\ft{\vec{W}} $ $= M^{-1} \ft{\vec{w}}$. 
By Remark \ref{XREMARK2}, we know that $\| \vec{u} \|_{X^{-\frac{1}{2},\frac{1}{2}}}$,  
$\| \vec{v} \|_{X^{-\frac{1}{2},\frac{1}{2}}}$, and $ \| \vec{w} \|_{X^{-\frac{1}{2},\frac{1}{2}}} $
are defined in terms of $(|\ft{U_1}|, |\ft{U_2}|)$, $(|\ft{V_1}|, |\ft{V_2}|)$, and$(|\ft{W_1}|, |\ft{W_2}|)$.
Therefore, without loss of generality, we can assume that the components of 
$\ft{\vec{U}}$, $\ft{\vec{V}}$, and $\ft{\vec{W}}$ are all nonnegative.
Moreover, since $\vec{u}, \vec{v},$ and $\vec{w} $ have mean 0 in $x$, we assume $\xi_1, \xi_2, \xi_3 \ne 0$.
The left hand side of \eqref{XDUALBILINEAR} can be written as 
\begin{align*} 
 \bigg| \iintt_{\substack{
\tau_1 + \tau_2 + \tau_3  = 0 \\
\xi_1 + \xi_2 + \xi_3 = 0
}}
\Big[  B\big( M(\xi_1)\ft{\vec{U}}(\xi_1, \tau_1), M(\xi_2)\ft{\vec{V}}(\xi_2, \tau_2) \big), 
\ M(\xi_3) \ft{\vec{W}} (\xi_3, \tau_3) \Big]_{\mathbb{R}^2}  d\xi_2^\ld d\xi_3^\ld d\tau_2 d\tau_3  \bigg| 
\end{align*}

\noindent
With $  M(\xi) = \mtx{\mu_1(\xi)}{\mu_2(\xi)}{\mu_3(\xi)}{\mu_4(\xi)}$, we have
\[=\bigg| \iintt_{\substack{
\tau_1 + \tau_2 + \tau_3  = 0 \\
\xi_1 + \xi_2 + \xi_3 = 0
}}
\frac{1}{2} \sum_{j, k, l = 1}^2 C_{j,k,l}(\xi_1, \xi_2, \xi_3)\ft{U_j}(\xi_1, \tau_1) \ft{V_k}(\xi_2, \tau_2) \ft{W_l}(\xi_3, \tau_3) 
d\xi_2^\ld d\xi_3^\ld d\tau_2 d\tau_3 \bigg|, \]

\noindent 
where
\begin{align*}
C_{j, k, l}(\xi_1, \xi_2, \xi_3)  & = \mu_j (\xi_1) \mu_{k+2}(\xi_2)\mu_{l+2}(\xi_3) \\
 & + \mu_{j+2}(\xi_1) \mu_k (\xi_2) \mu_{l+2}(\xi_3)  
 + \mu_{j+2}(\xi_1) \mu_{k+2}(\xi_2)\mu_l(\xi_3).  
\end{align*}

\noindent 
Since $M(\xi)$ is orthogonal, we have
$ |\mu_1(\xi)|^2 + |\mu_3(\xi)|^2 
= |\mu_2(\xi)|^2 + |\mu_4(\xi)|^2 = 1$.
In particular, we have $ |\mu_j(\xi)| \leq 1$ for all $\xi \in \mathbb{Z}/\ld$ and $  j = 1, 2, 3, 4 $.  
Thus, $|C_{j, k, l}(\xi_1, \xi_2, \xi_3)| \leq 3$ for all $\xi_1, \xi_2, \xi_3 \in \mathbb{Z}/\ld$ and $j, k, l = 1, 2$.
Hence, it is enough to prove, with $\xi_1, \xi_2, \xi_3 \neq 0$, 
\begin{align*} 
\iintt_{\substack{
\tau_1 + \tau_2 + \tau_3  = 0 \\
\xi_1 + \xi_2 + \xi_3 = 0
}}
 & \ft{U_j}(\xi_1, \tau_1) \ft{V_k}(\xi_2, \tau_2) \ft{W_l}(\xi_3, \tau_3)
 d\xi_2^\ld d\xi_3^\ld d\tau_2 d\tau_3  
 \lesssim \ld^{0+}  \| \vec{u} \|_{X^{-\frac{1}{2},\frac{1}{2}}}  
\| \vec{v} \|_{X^{-\frac{1}{2},\frac{1}{2}}}  \| \vec{w} \|_{X^{-\frac{1}{2},\frac{1}{2}}} ,
\end{align*}

\noindent 
for $ j, k, l = 1, 2$.
Recall that $d_1(\xi)$ and $d_2(\xi)$ are given by \eqref{XEIGENVALUE}
where $ L = \frac{1}{2} \sqrt{ p^2+ 4q^2}.$
Now, let $p_1$ and $ q_1$ denote the means of $u_0$ and $v_0$ for the original periodic problem on $[0, 2 \pi)$, i.e. with $ \ld = 1$.
Let $ L_1 = \frac{1}{2}  \sqrt{ p_1^2+ 4q_1^2}.$  
Then, by the scaling property discussed, 
we have $L = L_1 / \ld^2$.
Note that $L_1$ is really the constant given by each given initial value problem on $ [0, 2 \pi ) $
and $L$ depends on $\ld$ as shown above.

\noindent 
\textbf{{$\bullet$ Case (1):}} $|\xi_1|, |\xi_2|, |\xi_3| \geq \max( L_1 / \ld, 1) $

Since $\xi_1 + \xi_2 + \xi_3 = 0$, we have
$3 \xi_1 \xi_2 \xi_3 = \xi_1^3 + \xi_2^3 + \xi_3^3. $
Thus we have
\[ \big| \sum_{m = 1}^3 \tau_m - d_{j_m}(\xi_m) \big| 
 = \big| 3\xi_1 \xi_2 \xi_3 + L \big( (-1)^{j} \xi_1 + (-1)^{k} \xi_2 + (-1)^{l} \xi_3 \big) \big| .\]

\noindent 
For $(j_1, j_2, j_3) = (1, 1, 1)$, we have 
$ \big| \sum_{m = 1}^3 \tau_m - d_{j_m}(\xi_m) \big| 
=  3 |\xi_1 \xi_2\xi_3| $.
For $(j_1, j_2, j_3) = (1, 1, 2)$, we separate into two cases. 
If $\xi_1\xi_2<0$, then $| 3 \xi_1 \xi_2 - 2L| \geq |\xi_1 \xi_2| $ since $L \geq 0$.
Otherwise, i.e. if $\xi_1\xi_2 > 0$, then we have
$| 3 \xi_1 \xi_2 - 2L|  = \xi_1 \xi_2 + 2(\xi_1\xi_2 - L) \geq \xi_1 \xi_2$.

\noindent
since  \[ \xi_1 \xi_2 - L \geq 
\begin{cases}
\frac{L_1^2}{\ld^2} - \frac{L_1}{\ld^2} \geq 0 & \text{ if } L_1 \geq 1\\
1 - \frac{L_1}{\ld^2} \geq 1 - L_1\geq 0 & \text{ if } L_1 < 1.
\end{cases}
\]

\noindent
Therefore, we have
$| \sum_{m = 1}^3 \tau_m - d_{j_m}(\xi_m) | 
 =   |\xi_3 |\, | 3 \xi_1 \xi_2 - 2L|  
 \geq  |\xi_1\xi_2\xi_3| $.
All the other cases follow in a similar way and thus we have
$\big| \sum_{m = 1}^3 \tau_m - d_{j_m}(\xi_m) \big| 
 \gtrsim | \xi_1 \xi_2 \xi_3|$
for $ j_1, j_2, j_3 = 1, 2 $. 
Then, the rest follows as in \cite{CKSTT4}, using H\"older and Corollary \ref{COR:XL4}.

\noindent
\textbf{{$\bullet$ Case (2):}} $|\xi_1|, |\xi_2|, |\xi_3| \leq \max( L_1 / \ld, 1) $

We have $\jb{ \xi_1}, \jb{ \xi_2}, \jb{ \xi_3} \lesssim \max(L_1/\ld, 1) \leq \max(L_1, 1) \lesssim 1$ in this case.
Thus, the argument from \cite{CKSTT4} directly applies here, using H\"older and Corollary \ref{COR:XL4}.

 Lastly, since $\xi_1 + \xi_2 + \xi_3 = 0$, we  consider the case when one of the frequencies 
is small and the other two are large. Without loss of generality, we assume $ |\xi_3| $ is small, i.e.

\noindent \textbf{{$\bullet$ Case (3):}} 
$ 0 < |\xi_3| \leq \max( L_1 / \ld , 1) \leq |\xi_1|, |\xi_2| $

In this case, we have $|\xi_1 \xi_2| - L \geq 0 $ as in Case (1).
Since $\xi_3 \in \mathbb{Z}/\ld \setminus \{0\}$, we have $|\xi_3| \geq 1 / \ld$.
Thus, we have
$ |\xi_1 \xi_3| - L \geq \frac{L_1}{\ld} \frac{1}{\ld} - \frac{L_1}{\ld^2} = 0 $
Similarly, we have $|\xi_2 \xi_3 | - L \geq 0 $.  
Hence by repeating the computation in Case (1), we have
\[ 1 \lesssim |\xi_3|^{-\frac{1}{2}} \frac{ \jb{ \tau_1 - d_j(\xi_1)}^\frac{1}{2} +\jb{ \tau_2 - d_k(\xi_1)}^\frac{1}{2}
+ \jb{ \tau_3 - d_l(\xi_3)}^\frac{1}{2}} { \jb{\xi_1}^\frac{1}{2}\jb{\xi_2}^\frac{1}{2}\jb{\xi_3}^\frac{1}{2}} 
\ \text{ for } j, k, l = 1, 2. \]

\noindent
Then, the rest follows as in \cite{CKSTT4} using Lemmata \ref{XEMBED1}  and   \ref{XEMBED2}.

Next, we consider the estimate for the $L^2_\xi L^1_\tau$ part of the $Z^{-\frac{1}{2}}$ norm.
The basic idea is to reduce the proof to the previous case by applying Cauchy-Schwarz in $\tau$.
(See \cite{CKSTT4} and the proof of Proposition \ref{PROP:bilinear1}.)
Then, the main issue is when
there is no contribution from $\jb{\tau_1 - d_j(\xi_1)}$ or $\jb{\tau_2 - d_k(\xi_2)}$.
Suppose
$\jb{\tau_1 - d_j(\xi_1)} , \jb{\tau_2 - d_k(\xi_2)} \ll \jb{\tau - d_l(\xi)}^{\frac{1}{100}}$.
Also, assume
\begin{equation} \label{XNOTSMALL}
\max( |\xi|,|\xi_1|, |\xi_2|), \med( |\xi|,|\xi_1|, |\xi_2|) \geq \max(L_1/\ld, 1).
\end{equation}

\noindent
Let
$\G_{j, k, l}^\xi(\xi_1)  =   - d_l(\xi)  + d_j(\xi_1) + d_k(\xi_2)$.
Then, in this case, we have
\[\tau  - d_l(\xi)  =\G_{j, k, l}^\xi(\xi_1) + o(\jb{\tau  - d_l(\xi)}^\frac{1}{100}) 
 = \G_{j, k, l}^\xi(\xi_1) + o\big(|\G_{j, k, l}^\xi(\xi_1)|^\frac{1}{100}\big).\]

\noindent 
Let $\Omega_{j, k, l} (\xi)$ be a set defined by
\begin{multline*}
 \Omega_{j, k, l} (\xi) = \big\{ \eta \in \mathbb{R}  : \eta = \G_{j, k, l}^\xi(\xi_1) + o(|\G_{j, k, l}^\xi(\xi_1)|^\frac{1}{100})\\
 \text{ for some } \xi_1, \xi_2 \in \mathbb{Z}/\ld \text{ with } \xi = \xi_1 + \xi_2 \text{, satisfying \eqref{XNOTSMALL}} \big\} .
\end{multline*} 

\noindent
Then, the rest follows as in  \cite{CKSTT4} and 
Part 2 in the proof of Proposition \ref{PROP:bilinear1} 
thanks to the following lemma.

\begin{lemma} \label{XCLOSETOCURVE}
Fix $\xi \in \mathbb{Z} / \ld  \setminus \{0\} $. Then, for all dyadic $M \geq 1$, we have
\begin{equation} \label{ZTHINSET}
\left| \Omega_{j, k, l} (\xi) \cap \big\{ | \eta| \sim M \big\} \right| \lesssim \ld^\frac{3}{2} M^\frac{2}{3}
\end{equation}
for $j, k, l = 1, 2$.
\end{lemma}

\noindent
We omit the proof of this lemma since the proof is analogous to those 
of Lemma 7.4 in \cite{CKSTT4} and Lemma \ref{closetocubic11}.


\begin{thebibliography}{99}

\bibitem{AC} 	B. Alvarez-Samaniego,  X. Carvajal, {\it On the local well-posedness for some systems 
	of coupled KdV equations,}  Nonlinear Anal.  69  (2008),  no. 2, 692--715.

\bibitem{AR} V. Arnold, {\it Geometrical Methods in the Theory of Ordinary Differential Equations},
2nd ed., Springer-Verlag, New York, 1988.

\bibitem{ACW} J. M. Ash, J. Cohen,  G. Wang, {\it On strongly interacting internal solitary waves,} 
J. Fourier Anal. Appl.  2 (1996), no. 5, 507--517.

\bibitem{BP} D. Bambusi, S. Paleari, {\it Families of Periodic Solutions of Resonant PDEs, }
J. Nonlinear Sci. 11 (2001), 69--87. DOI:10.1007/s003320010010.

\bibitem{BOP1} D. Bekiranov, T. Ogawa, G Ponce, {\it Weak solvability and well-posedness of a coupled Schr\"odinger-Korteweg de Vries equation 
for capillary-gravity wave interactions, } Proc. AMS 125 (1997), 2907--2919.

\bibitem{BB}M. Berti, P. Bolle, {\it Cantor families of periodic solutions for completely resonant nonlinear wave equations,}
  Duke Math. J.  134  (2006),  no. 2, 359--419.

\bibitem{BPST} J. L. Bona, G. Ponce, J. C. Saut, and M. Tom, {\it A model system for strong interaction 
between internal solitary waves,} Comm. Math. Phys. 143 (1992), no. 2, 287--313.

\bibitem{BO1} J. Bourgain, {\it Fourier transform restriction phenomena for certain lattice subsets and applications to 
nonlinear evolution equations II}, Geom. Funct. Anal., 3 (1993), 209--262.

\bibitem{BO3} J. Bourgain, {\it Periodic Korteweg-de Vries equation with measures as initial data,}
Sel. Math., New Ser. 3 (1997) 115--159.


\bibitem{CCT} M. Christ, J. Colliander, T. Tao, {\it Asymptotics, frequency modulation, and low-regularity 
illposedness of canonical defocusing equations,}  Amer. J. Math.  125  (2003),  no. 6, 1235--1293.


\bibitem{CKSTT4}J. Colliander, M. Keel, G. Staffilani, H. Takaoka, T. Tao, {\it Sharp Global Well-Posedness for KdV and Modified KdV
on $\mathbb{R}$ and $\mathbb{T}$,}
J. Amer. Math. Soc. 16 (2003), no. 3, 705--749.


\bibitem{F} X. Feng, {\it Global well-posedness of the initial value problem for the Hirota-Satsuma system,}
Manu. Math. 84 (1994), 361-378.

\bibitem{GG} J.A. Gear, R. Grimshaw, {\it Weak and Strong interactions between internal solitary waves, }
Stud. Appl. Math. 70 (1984), no. 3, 235--258.


\bibitem{HS} R. Hirota, J. Satsuma, {\it Soliton solutions of a coupled Korteweg-de Vries equation,}
Partial Diff. Eq. 2 (1981), 408--409.

\bibitem{KT} T. Kappeler, P. Topalov, {\it Global wellposedness of KdV in $H\sp {-1}(\mathbb T,\mathbb R)$,} Duke Math. J. 135 (2006), no. 2, 327--360. 


\bibitem{KPV4} C. Kenig, G. Ponce,  L. Vega, {\it A bilinear estimate with applications to the KdV equation,}
J. Amer. Math. Soc. 9 (1996), no. 2 573--603.

\bibitem{KPV5} C. Kenig, G. Ponce,  L. Vega, {\it On the ill-posedness of some canonical dispersive equations,}
Duke Math. J. 106 (2001), no.3, 617--633.

\bibitem{KPV6} C. Kenig, G. Ponce,  L. Vega, {\it The Cauchy problem for the Korteweg-de Vries equation 
in Sobolev spaces of negative indices,} Duke Math. J. 71 (1993), no. 1, 1--21. 


\bibitem{LP} F. Linares, M. Panthee, {\it On the Cauchy problem for a coupled system of KdV equations,}
Comm. Pure Appl. Anal. 3 (2004), no. 3, 417--431.

\bibitem{MB} A. Majda, J. Biello, {\it The nonlinear interaction of barotropic and equatorial baroclinic Rossby waves,}
J. Atmospheric Sci. 60(2003), no. 15, 1809 --1821.


\bibitem{OHTHESIS} C. (T.) Oh, {\it Well-posedness theory of a one parameter family of coupled KdV-type systems 
and their invariant measures,} Ph.D. Thesis, University of Massachusetts Amherst (2007).

\bibitem{OH2} T. Oh, {\it Diophantine Conditions in Global Well-Posedness of  Coupled KdV-Type Systems,}
Electron. J. Diff. Eqns.,  Vol. 2009 (2009), No. 52, pp. 1-48.



\bibitem{OH3} T. Oh, {\it Invariant Gibbs Measures and a.s. Global Well-Posedness for Coupled KdV Systems,}
to appear in Diff. Integ. Equ.


\bibitem{OH4} T. Oh, {\it Invariance of the white noise for KdV,}
to appear in Comm. Math. Phys.


\bibitem{ST} J.C. Saut, N. Tzvetkov, {\it On a model system for the oblique interaction of internal gravity waves, }
M2AN Math. Model Numer. Anal. 34 (2000), no.2 501--523.
 


\bibitem{TZ} N. Tzvetkov, {\it Remark on the local ill-posedness for KdV equation,}
C.R. Acad. Sci. Paris, t. 329 S\'erie I (1999), 1043--1047.


\end{thebibliography}
\end{document}